\documentclass[12pt]{article}
\usepackage{tikz}
\usepackage{ucs}
\usepackage{amssymb}
\usepackage{amsthm}
\usepackage{amsmath}
\usepackage{latexsym}
\usepackage[cp1251]{inputenc}
\usepackage[english]{babel}
\usepackage{graphicx}
\usepackage{wrapfig}
\usepackage{caption}
\usepackage{subcaption}
\usepackage{txfonts}
\usepackage{lmodern}
\usepackage{mathrsfs}
\usepackage{algorithm}
\usepackage{multicol}
\usepackage{enumerate}
\usepackage{titlesec}
\usepackage{paralist}
\usepackage{mathtools}
\usepackage[colorlinks]{hyperref}
\usepackage{changes}

\titleformat{\paragraph}
{\normalfont\normalsize\bfseries}{\theparagraph}{1em}{}
\titlespacing*{\paragraph}
{0pt}{3.25ex plus 1ex minus .2ex}{1.5ex plus .2ex}

\newcommand*{\myproofname}{Proof}

\usepackage{wasysym}

\usepackage{fullpage}

\def\qed{\hfill\ifhmode\unskip\nobreak\fi\qquad\ifmmode\Box\else\hfill$\Box$\fi}

\title{Monochromatic connected matchings\\ in 2-edge-colored multipartite graphs }
\date{\today}
\author{
J\' ozsef Balogh~\thanks{Department of Mathematics, University of Illinois at Urbana--Champaign, IL, USA and Moscow Institute of Physics and Technology, Russian Federation, jobal@illinois.edu. 
Research of this author is partially supported by  NSF Grants DMS-1500121 and DMS-1764123, Arnold O. Beckman Research Award (UIUC) Campus Research Board 18132 and the Langan Scholar Fund (UIUC).}
\and Alexandr Kostochka \thanks{Department of Mathematics, University of Illinois at Urbana--Champaign, IL, USA and
Sobolev Institute of Mathematics, Novosibirsk 630090, Russia, kostochk@math.uiuc.edu. Research of this author is supported in part by NSF grant
 DMS-1600592, by UIUC Campus Research Board Award RB20003, and by grant  19-01-00682 of the Russian Foundation for Basic Research.}
  \and Mikhail Lavrov\thanks{Department of Mathematics, Kennesaw State University, GA, USA, mlavrov@kennesaw.edu; the work was done while M. Lavrov was a postdoc at Department of Mathematics, University of Illinois at Urbana--Champaign.}
 \and Xujun Liu\thanks{Department of Foundational Mathematics, Xi'an Jiaotong-Liverpool University, Jiangsu Province, China,  	Xujun.Liu@xjtlu.edu.cn. The work was partially done while X. Liu was a PhD student at Department of Mathematics, University of Illinois at Urbana--Champaign. Research of this author was supported in part by Award RB17164 of the UIUC Campus Research Board.}
 }

\makeatletter
\newcommand{\neutralize}[1]{\expandafter\let\csname c@#1\endcsname\count@}
\makeatother

\newtheorem{theorem}{Theorem}
\newtheorem{prop}[theorem]{Proposition}
\newtheorem{lemma}[theorem]{Lemma}
\newtheorem{conj}[theorem]{Conjecture}
\newtheorem{claim}[theorem]{Claim}

\newtheorem{defn}[theorem]{Definition}

\newtheorem{remk}[theorem]{Remark}

	\def\quotient#1#2{%
		\raise1ex\hbox{$#1$}\Big/\lower1ex\hbox{$#2$}%
	}

	\renewcommand{\epsilon}{\varepsilon}

\begin{document}
	\maketitle
	
\begin{abstract}
A matching $M$ in a graph $G$ is {\em connected} if all the edges of $M$ are in the same component of $G$. Following \L uczak,
there have been many results using
 the existence of  large connected matchings in cluster graphs  with respect to regular partitions of large graphs to show the existence of long paths and other structures in these  graphs. We prove exact 
 Ramsey-type bounds on the sizes of monochromatic connected matchings in $2$-edge-colored multipartite graphs. In addition, we prove a stability theorem for such matchings.
\end{abstract}

	\section{Introduction}	
Recall that for graphs $G_0,\ldots,G_k$ we write $G_0 \mapsto (G_1,\ldots,G_k)$ if for every $k$-coloring of the edges of $G_0$, for some $i\in [k]$ there will be a copy of $G_i$ with all edges of color $i$.	The {\em Ramsey number} $R_k(G)$ is the minimum $N$ such that $K_N\mapsto (G_1,\ldots, G_k)$, where $G_1=\ldots=G_k=G$.
Gerencs\' er and Gy\' arf\' as~\cite{GG1} proved in 1967 that the $n$-vertex path $P_n$ satisfies $R_2(P_n)=\left\lfloor \frac{3n-2}{2}\right\rfloor$.
Significant progress in bounding $R_k(P_n)$ for  $k\geq 3$ and $R_k(C_n)$ for even $n$  was achieved after 2007~(see~\cite{BLSSW,BS1, DK1,FL1,GRSS1,
GRSS2,KS1,LSS1,S1} and some references in them). All these proofs used the Szemer\' edi Regularity Lemma~\cite{Sz} and the idea of connected matchings in regular partitions due to \L uczak~\cite{L1}. 
	
Recall that a matching $M$ in a graph $G$ is {\em connected} if all the edges of $M$ are in the same component of $G$. 
We will denote a connected matching with $k$ edges by $M_k$.
The use of connected matchings is illustrated for example by the following version of a lemma by  Figaj and  \L uczak~\cite{FL1}.

\begin{lemma}[Lemma 8 in~\cite{LSS1} and Lemma~1 in~\cite{KS1}]\label{L1}
Let a real number
$c > 0$ and a positive integer $k$
be given. If for every $\epsilon>0$ there exists a
$\delta>0$ and an $n_0$ such that for every even
$n > n_0$ and each graph $G$ with $v(G)>(1 +\epsilon) cn$ and
$e(G)\geq (1-\delta){v(G)\choose 2}$ each $k$-edge-coloring of $G$
has a monochromatic connected matching $M_{n/2}$, then for sufficiently large $n$, $R_k(C_n)\leq(c+o(1))n$ (and hence
$R_k(P_n)\leq(c+o(1))n$).
\end{lemma}

Moreover,  Figaj and  \L uczak~\cite{FL1} showed that for any real positive numbers $\alpha_1, \alpha_2, \alpha_3$ the Ramsey number for a triple of even cycles of lengths $2 \lfloor \alpha_1 n \rfloor, 2 \lfloor \alpha_2 n \rfloor, 2 \lfloor \alpha_3 n \rfloor,$ respectively, is $(\alpha_1+ \alpha_2+\alpha_3 +\max\{\alpha_1, \alpha_2, \alpha_3\} + o(1))n.$

Similar problems with complete $3$-partite host graphs $K_{N,N,N}$ and complete bipartite host graphs $K_{N,N}$ instead of $K_N$ were considered by 
Gy\' arf\' as,  Ruszink\' o,   S\'ark\"ozy and  Szemer\' edi~\cite{GRSS0}, DeBiasio
and  Krueger~\cite{DK1} and Bucic,  Letzter and  Sudakov~\cite{BLS1,BLS2}. All of these papers also exploited connected matchings in cluster graphs.
The main result of Gy\' arf\' as,  Ruszink\' o,   S\' ark\" ozy and  Szemer\' edi~\cite{GRSS0} was

\begin{theorem}[\cite{GRSS0}]\label{tt1} For positive integers $n$,  $K_{n,n,n}\mapsto (P_{2n-o(n)},P_{2n-o(n)})$.
\end{theorem}

They also conjectured the exact bound:

\begin{conj}[\cite{GRSS0}]\label{cc1} For positive integers $n,$  $K_{n,n,n}\mapsto (P_{2n+1},P_{2n+1})$.
\end{conj}	

Since the papers~\cite{GRSS0,BLS1,BLS2} proved asymptotic bounds, they  used approximate bounds
on maximum sizes of monochromatic connected matchings in edge-colored dense
 multipartite graphs. But for the exact bound~\cite{GRSS1,GRSS2} (for large $N$) on long paths in $3$-edge-colored $K_N$ and for the exact bound by
DeBiasio
and  Krueger~\cite{DK1} on long paths and cycles in $2$-edge-colored bipartite graphs, one needs a stability theorem: {\em either the edge-colored graph has a  large monochromatic connected matching, or the edge-coloring is very special}.

In this paper, we find exact bounds on the size of a maximum monochromatic connected matching in each $2$-edge-colored complete multipartite graph $K_{n_1,\ldots,n_k}$. This generalizes, sharpens and extends the corresponding results in~\cite{GRSS0} and can be considered as an extension of one of the  results in~\cite{DK1}. We also prove a corresponding stability theorem in the spirit of~\cite{GRSS1} and~\cite{DK1}.
In our follow-up paper~\cite{BKLL1} we use this stability theorem to prove among other results that for large $n$,  Conjecture~\ref{cc1} and the relation $K_{n,n,n}\mapsto (C_{2n},C_{2n})$ hold.

\section{Notation and results}	

Let $\alpha'(G)$ denote the size of a largest matching in $G$ and $\alpha'_*(G)$ denote the size of a largest connected matching in $G$. Let $\alpha(G)$ denote the independence number and $\beta(G)$ denote the size of a smallest vertex cover in $G$.

For a graph $G$ and $W_1, W_2 \subseteq V(G)$, let $G[W_1, W_2]$ denote the subgraph of $G$ consisting of edges with one endpoint in $W_1$ and 
the other endpoint in $W_2$.


 We seek  minimal restrictions on $n_1\geq n_2\geq \ldots\geq n_s$ guaranteeing that every $2$-edge-coloring of
$K_{n_1, n_2, \ldots , n_s}$ contains a monochromatic  $M_n$. An obvious necessary condition is that 
\begin{equation}\label{jj1}
 N:=n_1+\ldots+n_s\geq 3n-1.
 \end{equation}
  Indeed, even $K_{3n-2}\not\mapsto (M_{n},M_{n})$: for $G=K_{3n-2}$, partition $V(G)$
into sets $U_1$ and $U_2$ with $|U_1|=2n-1$, $|U_2|=n-1$, and color the edges of $G[U_1,U_2]$ with red and the rest of the edges with blue. Then there is no monochromatic  $M_n$; see Figure~\ref{ex1}. The other natural requirement is that 
\begin{equation}\label{jj2}
 N-n_1=n_2+\ldots+n_s\geq 2n-1.
 \end{equation}
 Indeed, for $N=n_1+2n-2$, consider the graph $H$ obtained from $K_N$ by deleting the edges inside a vertex subset $U_1$ with $|U_1|=n_1$. Graph $H$ contains every $K_{n_1, n_2, \ldots , n_s}$ with $n_2+\ldots+n_s=2n-2$.
 Partition $V(H)-U_1$ into sets $U_2$ and $U_3$ with $|U_2|=|U_3|=n-1$. Color all edges incident with $U_2$ red, and the remaining edges of $H$ blue. Again,  there is no monochromatic  $M_n$; see Figure~\ref{ex2}.

\begin{figure}[ht]\label{f1}
\hspace{5mm}
\begin{minipage}[b]{0.4\textwidth}
\begin{tikzpicture}[scale=0.5, transform shape]

\draw  (-2.5,3.5) [black, fill=blue!10] ellipse (4 and 2);
\draw  (-2.5,-1) [black, fill=blue!10] ellipse (2 and 1);
\node [ shape=circle, minimum size=0.1cm,  fill = black!1000, align=center] (v1) at (-1.95,3.5) {};
\node [ shape=circle, minimum size=0.1cm,  fill = black!1000, align=center] (v2) at (-0.8,3.5) {};
\node [ shape=circle, minimum size=0.1cm,  fill = black!1000, align=center] (v3) at (0.35,3.5) {};
\node [ shape=circle, minimum size=0.1cm,  fill = black!1000, align=center] (v4) at (-3.2,3.5) {};
\node [ shape=circle, minimum size=0.1cm,  fill = black!1000, align=center] (v5) at (-4.35,3.5) {};
\node [ shape=circle, minimum size=0.1cm,  fill = black!1000, align=center] (v6) at (-5.5,3.5) {};
\node [ shape=circle, minimum size=0.1cm,  fill = black!1000, align=center] (v7)  at (-3.5,-1) {};
\node [ shape=circle, minimum size=0.1cm,  fill = black!1000, align=center] (v8)  at (-1.5,-1) {};
\node[ shape=circle, minimum size=0.1cm,  fill = black!1000, align=center] (v10)  at (-2.5,-1) {};
\draw [red, line width = 2]  (v6) edge (v7);
\draw [red, line width = 2] (v6) edge (v10);
\draw [red, line width = 2] (v6) edge (v8);
\draw[red, line width = 2]  (v5) edge (v7);
\draw [red, line width = 2] (v5) edge (v10);
\draw  [red, line width = 2](v5) edge (v8);
\draw [red, line width = 2] (v4) edge (v7);
\draw [red, line width = 2] (v4) edge (v10);
\draw [red, line width = 2] (v4) edge (v8);

\draw  [red, line width = 2](v1) edge (v7);
\draw  [red, line width = 2](v1) edge (v10);
\draw [red, line width = 2] (v1) edge (v8);
\draw [red, line width = 2] (v2) edge (v7);
\draw [red, line width = 2] (v2) edge (v10);
\draw [red, line width = 2] (v2) edge (v8);
\draw [red, line width = 2] (v3) edge (v7);
\draw [red, line width = 2] (v3) edge (v10);
\draw [red, line width = 2] (v3) edge (v8);

\node at (-9,3.5) {\LARGE{$|U_1|=2n-1$}};
\node at (-8.5,-1) {\LARGE{$|U_2|=n-1$}};

\end{tikzpicture}

\caption{Example for condition~\eqref{jj1}.}\label{ex1}
\end{minipage}
\begin{minipage}[b]{0.5\textwidth}
\begin{tikzpicture}[scale=0.5, transform shape]

\draw  (-2.5,3.5) ellipse (4 and 2);
\draw [black, fill=red!80, rotate = 45] (-6,2.5) ellipse (1 and 2);
\draw [black, fill=blue!10, rotate = -45] (2.5,-1) ellipse (1 and 2);

\node [ shape=circle, minimum size=0.1cm,  fill = black!1000, align=center] (v1) at (-1.9,3.5) {};
\node [ shape=circle, minimum size=0.1cm,  fill = black!1000, align=center] (v2) at (-0.7,3.5) {};
\node [ shape=circle, minimum size=0.1cm,  fill = black!1000, align=center] (v3) at (0.5,3.5) {};
\node [ shape=circle, minimum size=0.1cm,  fill = black!1000, align=center] (v4) at (-3.25,3.5) {};
\node [ shape=circle, minimum size=0.1cm,  fill = black!1000, align=center] (v5) at (-4.35,3.5) {};
\node [ shape=circle, minimum size=0.1cm,  fill = black!1000, align=center] (v6) at (-5.5,3.5) {};
\node [ shape=circle, minimum size=0.1cm,  fill = black!1000, align=center] (v7) at (-5,-3.5) {};
\node [ shape=circle, minimum size=0.1cm,  fill = black!1000, align=center] (v8) at (-7,-1.5) {};
\node[ shape=circle, minimum size=0.1cm,  fill = black!1000, align=center] (v10) at (-6,-2.5) {};

\node [ shape=circle, minimum size=0.1cm,  fill = black!1000, align=center] (v11) at (0,-3.5) {};
\node [ shape=circle, minimum size=0.1cm,  fill = black!1000, align=center] (v12) at (2,-1.5) {};
\node[ shape=circle, minimum size=0.1cm,  fill = black!1000, align=center] (v13) at (1,-2.5) {};

\draw [blue!30, line width = 2]  (v6) edge (v11);
\draw [blue!30, line width = 2] (v6) edge (v13);
\draw [blue!30, line width = 2] (v6) edge (v12);
\draw[blue!30, line width = 2]  (v5) edge (v11);
\draw [blue!30, line width = 2] (v5) edge (v13);
\draw  [blue!30, line width = 2](v5) edge (v12);
\draw [blue!30, line width = 2] (v4) edge (v11);
\draw [blue!30, line width = 2] (v4) edge (v13);
\draw [blue!30, line width = 2] (v4) edge (v12);

\draw  [blue!30, line width = 2](v1) edge (v11);
\draw  [blue!30, line width = 2](v1) edge (v13);
\draw [blue!30, line width = 2] (v1) edge (v12);
\draw [blue!30, line width = 2] (v2) edge (v11);
\draw [blue!30, line width = 2] (v2) edge (v13);
\draw [blue!30, line width = 2] (v2) edge (v12);
\draw [blue!30, line width = 2] (v3) edge (v11);
\draw [blue!30, line width = 2] (v3) edge (v13);
\draw [blue!30, line width = 2] (v3) edge (v12);

\draw [red!100, line width = 2]  (v6) edge (v7);
\draw [red!100, line width = 2] (v6) edge (v10);
\draw [red!100, line width = 2] (v6) edge (v8);
\draw[red!100, line width = 2]  (v5) edge (v7);
\draw [red!100, line width = 2] (v5) edge (v10);
\draw  [red!100, line width = 2](v5) edge (v8);
\draw [red!100, line width = 2] (v4) edge (v7);
\draw [red!100, line width = 2] (v4) edge (v10);
\draw [red!100, line width = 2] (v4) edge (v8);

\draw  [red!100, line width = 2](v1) edge (v7);
\draw  [red!100, line width = 2](v1) edge (v10);
\draw [red!100, line width = 2] (v1) edge (v8);
\draw [red!100, line width = 2] (v2) edge (v7);
\draw [red!100, line width = 2] (v2) edge (v10);
\draw [red!100, line width = 2] (v2) edge (v8);
\draw [red!100, line width = 2] (v3) edge (v7);
\draw [red!100, line width = 2] (v3) edge (v10);
\draw [red!100, line width = 2] (v3) edge (v8);

\node at (4,3.5) {\LARGE{$|U_1|=n_1$}};
\node at (-9.5,-3) {\LARGE{$|U_2|=n-1$}};
\node at (4.5,-3) {\LARGE{$|U_3|=n-1$}};

\draw [red!100, line width = 2]  (v8) edge (v12);
\draw   [red!100, line width = 2](v8) edge (v13);
\draw [red!100, line width = 2]  (v8) edge (v11);
\draw  [red!100, line width = 2] (v10) edge (v12);
\draw  [red!100, line width = 2] (v10) edge (v13);
\draw  [red!100, line width = 2] (v10) edge (v11);
\draw  [red!100, line width = 2] (v7) edge (v12);
\draw  [red!100, line width = 2] (v7) edge (v13);
\draw  [red!100, line width = 2] (v7) edge (v11);

\end{tikzpicture}

\caption{Example for condition~\eqref{jj2}.}\label{ex2}

\end{minipage}
\end{figure}

Our first main result is that the necessary conditions~(\ref{jj1}) and~(\ref{jj2}) together are sufficient for
 $K_{n_1, n_2, \ldots , n_s}\mapsto (M_{n},M_{n})$. We  prove it in the following more general form.

\begin{theorem}\label{2conn}
Let $x_1 \ge x_2 \ge 1, s\geq 2$, and let $G$ be a complete $s$-partite graph $K_{n_1,\ldots,n_s}$ such that
\begin{equation}\label{2conn-1}
N:=n_1+\ldots+n_s\geq 2x_1 + x_2 - 1,
\end{equation}
and
\begin{equation}\label{2conn-2}
N-n_i\geq x_1 + x_2 -1\quad  \mbox{for every $1\leq i\leq s$}.
\end{equation}
Let $E(G)=E_1\cup E_2$ be a partition of the edges of $G$, and let $G_i=G[E_i]$ for $i=1,2$. Then for some $i$, $\alpha'_*(G_i) \ge x_i$.
\end{theorem}

There are at least two types of $3$-edge-colorings of $K_{4n-3}$ with no monochromatic $M_n$.
We  use Theorem~\ref{2conn}  to show the following generalization of the   existence of a monochromatic connected matching
$M_n$ in each 3-edge-coloring of $K_{4n-2}$.

\begin{theorem}\label{3conn}
Let $1 \le x_2, x_3 \le x_1$,  $N = 2x_1 + x_2 + x_3 - 2$, and $G = K_N$.
\\
Let $E(G) = E_1 \cup E_2 \cup E_3$ be a partition of the edges of $G$, and let $G_i = G[E_i]$ for $i=1,2,3$. Then for some $i$, $\alpha'_*(G_i) \ge x_i$.
\end{theorem}

Finally, for the case $x_1 = x_2 = n$ of Theorem~\ref{2conn}, we prove a stability result which will be used in~\cite{BKLL1} to prove Conjecture~\ref{cc1} for large $N$. This will require a few definitions to state.

\begin{defn}\label{suitable}
For $\epsilon>0$ and $s\geq 2$, an $N$-vertex $s$-partite graph $G$ with parts $V_1,\ldots,V_s$ of sizes $n_1\geq n_2\geq\ldots\geq n_s$, and a 2-edge-coloring $E = E_1 \cup E_2$, is $(n,s,\epsilon)$-{\em suitable} if the following conditions hold:
\begin{equation}\label{suit-s1}\tag{S1}
	N=n_1+\ldots +n_s\geq 3n-1,
\end{equation}
\begin{equation}\label{suit-s2}\tag{S2}
	n_2+n_3+\ldots +n_s\geq 2n-1,
\end{equation}
and if $\widetilde{V}_i$ is the set of vertices in $V_i$ of degree at most $N-\epsilon n-n_i$ and $\widetilde{V}=\bigcup_{i=1}^s \widetilde{V}_i$, then 
\begin{equation}\label{suit-s3}\tag{S3}
	|\widetilde{V}|=|\widetilde{V}_1|+\ldots +|\widetilde{V}_s|<\epsilon n.
\end{equation}
We do not require $E_1 \cap E_2 = \emptyset$; an edge can have one or both colors. We write $G_i = G[E_i]$ for $i=1,2$.
\end{defn}

\begin{remk}
Note that a $2$-edge-coloring is actually not needed in Definition~\ref{suitable}. However, since we always talk about $(n, s, \epsilon)$-suitable graphs with a $2$-edge-coloring, we assume by default that an $(n, s, \epsilon)$-suitable graph has a $2$-edge-coloring and thus include it in the definition.
\end{remk}

Our stability result gives a partition of the vertices of near-extremal graphs called a $(\lambda,i,j)$-{\em bad partition}. There are two types of bad partitions for $(n,s,\epsilon)$-suitable graphs.

\begin{defn}
For $i\in \{1,2\}$, $\lambda > 0$, and an $(n,s,\epsilon)$-suitable graph $G$, a partition $V(G)=W_1\cup W_2$ of $V(G)$ is $(\lambda,i,1)$-{\em bad} if the following holds:
\begin{enumerate}
\item[(i)] $(1-\lambda)n \le |W_2| \le (1+\lambda)n_1$;
\item[(ii)] $|E(G_i[W_1,W_2])| \le \lambda n^2$;
\item[(iii)] $|E(G_{3-i}[W_1])| \le \lambda n^2$.
\end{enumerate}
\end{defn}

\begin{defn}
For $i\in \{1,2\}$, $\lambda > 0$, and an $(n,s,\epsilon)$-suitable graph $G$, a partition $V(G)=V_j\cup U_1\cup U_2$, $j \in [s]$, of $V(G)$ is $(\lambda,i,2)$-{bad} if the following holds:
\begin{enumerate}
\item[(i)] $|E(G_i[V_j,U_1])| \le \lambda n^2$;
\item[(ii)] $|E(G_{3-i}[V_j,U_2])| \le \lambda n^2$;
\item[(iii)] $n_j=|V_j| \ge (1-\lambda)n$;
\item[(iv)] $(1-\lambda)n \le |U_1| \le (1+\lambda)n$;
\item[(v)] $(1-\lambda)n \le |U_2| \le (1+\lambda)n$.
\end{enumerate}
\end{defn}

Our stability theorem is:

\begin{theorem}\label{stability}
Let $n\geq s\geq 2$,  $0<\epsilon< 10^{-3}\gamma<10^{-6}$ and $n>100/\gamma$. Let $G$ be an $(n,s,\epsilon)$-{suitable}  graph.
If $\max\{\alpha'_*(G_1), \alpha'_*(G_2)\}\le n(1+\gamma)$, then 
for some $i\in [2]$ and $j \in [2]$, $V(G)$ has a $(68\gamma,i,j)$-bad partition.
\end{theorem}

In the next section, we remind the reader of the notion and properties of the  Gallai--Edmonds decomposition, and in each of  the next three sections we prove one of the Theorems~\ref{2conn},~\ref{3conn} and~\ref{stability}.

\begin{remk}
One of the referees found a nicer and shorter proof for Theorem 4, using induction. Furthermore, the referee pointed out that a year after we have submitted our paper, Letzter~\cite{Letzter} introduced a method that could have shortened some of the proofs. In both cases, we preferred to stick to the original proofs.
\end{remk}

\section{Tools from graph theory}
We make extensive use of the Gallai--Edmonds decomposition  (called below {\em the GE-decompo\-sition} for short) of a graph $G$, defined below.

\begin{defn}
In a graph $G$, let $B$ be the set of vertices that are covered by every maximum matching in $G$. Let $A$ be the set of vertices in $B$ having at least one neighbor outside $B$, let $C = B - A$, and let $D = V(G) - B$. The {\em  GE-decomposition} of $G$ is the partition of $V(G)$ into the three sets $A, C, D$.
\end{defn}

\begin{defn}
A graph $G$ is {\em factor-critical} if $G$ has no perfect matching but for each $v \in V$, $G-v$ has a perfect matching. A {\em near-perfect matching} is a matching in which a single vertex is left unmatched.
\end{defn}

Edmonds and Gallai described important properties of this decomposition: 

\begin{theorem}[Gallai--Edmonds Theorem; Theorem 3.2.1 in \cite{LP}]\label{GE}
Let $A,C,D$ be the  GE-decomposition of a graph $G$. Let $G_1, \ldots, G_k$ be the components of $G[D]$. If $M$ is a maximum matching in $G$, then the following properties hold:
\begin{enumerate}
\item[(a)] $M$ covers $C$ and matches $A$ into distinct components of $G[D]$.
\item[(b)] Each $G_i$ is factor-critical and has a near-perfect matching in $M$.
\item[(c)] If $\emptyset \neq S \subseteq A$, then $N(S)$ intersects at least $|S|+1$ of $G_1, \ldots, G_k$.
\end{enumerate}
\end{theorem}

For bipartite graphs, we use the simpler K\"{o}nig--Egerv\'{a}ry theorem, which we apply in two equivalent forms:
\begin{theorem}[K\"{o}nig--Egerv\'{a}ry Theorem; Theorem 1.1.1 in \cite{LP}]\label{konig-egervary}
In a bipartite graph, the number of edges in a maximum matching is equal to the number of vertices in a minimum vertex cover.

Equivalently, if  $H$ is a bipartite graph with bipartition $(U,V)$, then 
\[
	\alpha'(H) = \min_{U_1 \subseteq U}\{|U| - |U_1| + |N(U_1)|\}.
\]
\end{theorem}




\section{Connected matchings in 2-edge-colorings (Theorem~\ref{2conn})}
In this section, we shall prove Theorem~\ref{2conn}. Let $G$ be  a complete $s$-partite graph $K_{n_1,\ldots,n_s}$ satisfying~\eqref{2conn-1} and~\eqref{2conn-2}. Let $V_1,\ldots,V_s$ be the parts of $G$ with $|V_i|=n_i$ for $i=1,\ldots,s$. 

We proceed by contradiction, assuming that there is a partition $E(G)=E_1\cup E_2$ such that 
\begin{equation}
	\alpha'_*(G_1) < x_1 \text{ and } \alpha'_*(G_2) < x_2. \label{no-cm}
\end{equation}
Among such edge partitions, we will find partitions with additional restrictions and study their properties. Eventually we will prove that such partitions do not exist. 

\subsection{Structure of $G$}

Among all $G$ and partitions $E(G)=E_1\cup E_2$ satisfying~\eqref{2conn-1}, \eqref{2conn-2} and~\eqref{no-cm}, choose one with the smallest $N$.

\begin{claim}\label{minN}
If $n_1\geq n_2\geq\ldots\geq n_s$, then either 
 $N=2x_1 + x_2 - 1$  or we have  $n_1=n_2$ and $N\leq 2x_1 + 2x_2 - s$.
\end{claim}
\begin{proof} Suppose $N>2x_1 + x_2 - 1$ and $v\in V_1$. Let $G'=G-v$. Then~\eqref{2conn-1}  and~\eqref{no-cm} hold for $G'$.
Hence by the minimality of $G$,~\eqref{2conn-2} does not hold for $G'$. 
Since~\eqref{2conn-2} does hold for $G$,  we conclude that  $n_1=n_2$ and $N-n_1=x_1 + x_2 - 1$. The last equality implies that $n_2=(x_1+x_2-1)-n_3-\ldots-n_s\leq x_1 +x_2 + 1 -s$. Hence
\[
	N=n_1+(N-n_1)=n_2+(x_1+x_2-1)\leq 2x_1 + 2x_2 - s,
\]
as claimed.
\end{proof}

\begin{claim}\label{s=2}
$G$ is not bipartite; that is, $s\geq 3$.
\end{claim}
\begin{proof} Suppose $s=2$. Then by~\eqref{2conn-2}, $n_1 = N - n_2 \geq x_1+x_2-1$ and $n_2 = N - n_1 \geq x_1+x_2-1$. It is sufficient to consider the situation that
$n_1=n_2=x_1+x_2-1$.

Suppose that for some $i \in \{1,2\}$, $G_i$ has at most one non-trivial component, i.e., $\alpha'(G_i) = \alpha'_*(G_i)$ (and so by \eqref{no-cm}, $\alpha'(G_i)<x_i$). By Theorem~\ref{konig-egervary}, $G_i$ has a vertex cover $C$ with $|C| \le x_i - 1$. Hence all edges of $G$ connecting $V_1 - C$ with $V_2 - C$ are in $E_{3-i}$. Thus $G_{3-i}$ contains $K_{x_1+x_2-1-|C|,x_1+x_2-1-|C|}$, which in turn contains $K_{x_{3-i}, x_{3-i}}$. Therefore $\alpha'_*(G_{3-i}) \ge x_{3-i}$, contradicting~\eqref{no-cm}.

Therefore $\alpha'(G_i) > \alpha'_*(G_i)$ for both $i \in \{1,2\}$. This means that each of $G_1$ and $G_2$ has more than one  nontrivial component. Let $A$ be the vertex set of one nontrivial component in $G_2$ and $B=(V_1\cup V_2)-A$. For each $i\in\{1,2\}$, let $A_i=V_i\cap A$, $B_i=V_i\cap B$, $a_i=|A_i|$, and $b_i=|B_i|$. 

Then for both $i\in\{1,2\}$, $G_1[A_i\cup B_{3-i}]=K_{a_i,b_{3-i}}$. So if there is at least one edge connecting $A_1$ with $A_2$ or $B_1$ with $B_2$ in $G_1$, then $G_1$ is connected and so $\alpha'_*(G_1)=\alpha'(G_1)$, a contradiction. Thus, $G_2[A_1\cup A_2]=K_{a_1,a_{2}}$ and $G_2[B_1 \cup B_2]=K_{b_1,b_2}$.

This means that $\min\{a_1, a_2\} < x_2$ and $\min\{b_1, b_2\} < x_2$. By the symmetry between $a_1$ and $a_2$, we may assume $a_1 < x_2$. Then $b_1 = (x_1 + x_2-1) - a_1 \ge x_1 \ge x_2$. Hence $b_2 < x_2$, and $a_2 = (x_1 + x_2 - 1) - b_2 \ge x_1$. But $G_1$ contains $K_{b_1, a_2}$, so it contains $K_{x_1,x_1}$, a contradiction to~\eqref{no-cm}.
\end{proof}

\subsection{Components of $G_i$}

Next, by analyzing the  components of $G_1$ and $G_2$, we will reduce the problem to a case where $G_1$ and $G_2$ have at most one nontrivial component each. Then it will be enough to find a large matching in either $G_1$ or $G_2$; the matching will automatically be connected, which will contradict assumption~\eqref{no-cm}.

\begin{claim}\label{compl}
For each of  $i \in \{1,2\}$, if $G_i$ is disconnected, then $\alpha'_*(G_{3-i})=\alpha'(G_{3-i})$.
\end{claim}
\begin{proof} Suppose $G_1$ is disconnected (the proof for the case when $G_2$ is disconnected is similar). Let $W_1$ induce a  component of $G_1$ and $W_2=V(G)-W_1$. We consider three cases:

\medskip\noindent\textbf{Case 1:} 
For some $j\in [s]$, $W_1\subseteq V_j$. Since $V_j$ is independent, $W_1=\{v\}$ for some $v\in V_j$.
Then all vertices in $V(G_2)-V_j$ are adjacent to  $v$ in $G_2$. So, $G_2$ has a component
$D$ containing
$V(G_2)-V_j+v$. Since $V_j$ is independent, every edge in $G_2$ has a vertex in $V(G)-V_j$, and hence lies in $D$.

\medskip\noindent\textbf{Case 2:}
For some distinct $j_1,j_2\in [s]$, $W_1\subseteq V_{j_1}\cup V_{j_2}$ and $W_1$ has a vertex $v_1\in V_{j_1}$ and a vertex $v_2\in V_{j_2}$.  By Claim~\ref{s=2},
$V(G)-V_{j_1}-V_{j_2}\neq\emptyset$, and by the case, each vertex in $V(G)-V_{j_1}-V_{j_2}$ is adjacent in $G_2$ to both $v_1$ and $v_2$.
Thus, a component $D$ of $G_2$ contains $W_1\cup (V(G)-V_{j_1}-V_{j_2})$. Furthermore, each vertex in $V_{j_1}-W_1$ 
is adjacent in $G_2$ to $v_2$, and each vertex in $V_{j_2}-W_1$
is adjacent in $G_2$ to $v_1$. It follows that $G_2$ is connected.

\medskip\noindent\textbf{Case 3:}
For some distinct $j_1,j_2,j_3\in [s]$, $W_1$  has a vertex $v_\ell\in V_{j_\ell}$ for all $\ell\in [3]$. Then each vertex in $W_2$
 is adjacent in $G_2$ to at least two of $v_1,v_2$ and $v_3$. Thus, a component $D$ of $G_2$ contains 
 $W_2$. If each $v\in W_1$ has in $G_2$ a neighbor in $W_2$, then $D=V(G)$, i.e. $G_2$ is 
 connected. Suppose there is $v\in W_1$  that has no neighbors in $W_2$ in $G_2$. We may assume $v\in V_{j_1}$.
 Then $W_2\subset V_{j_1}$. This means all vertices in $V(G)-D$ are in $V_{j_1}$.
Since $V_{j_1}$ is independent, every edge in $G_2$ has a vertex in $V(G)-V_{j_1}$, and hence lies in $D$.
\end{proof}

Claim~\ref{compl} implies that $\alpha'_*(G_i) = \alpha'(G_i)$ holds for at least one $i$. This equality does not necessarily hold for both $i=1$ and $i=2$, but we show that it is enough to prove Theorem~\ref{2conn} in the case where it does.

\begin{claim}\label{best-coloring}
If there are  partitions $E(G)=E_1\cup E_2$ of $E(G)$ such that  $G_1:=G[E_1]$ and $G_2=G[E_2]$ satisfy~\eqref{no-cm}, then 
some such partition in addition  satisfies all of the following:
\begin{itemize}
\item $\alpha'_*(G_1) = \alpha'(G_1)$ and $\alpha'_*(G_2) = \alpha'(G_2)$;
\item $G_1$ has the  GE-decomposition $(A,C,D)$ such that if $D_0 = C$ and $D_1, D_2, \dots, D_k$ are the components of $G_1[D]$ with $|D_1| \ge |D_2| \ge \dots \ge |D_k|$, then $G_1 - A$ has at least three components, and $G_2[D_j]$ is empty for $j=0,1,\dots,k$.
\end{itemize}
\end{claim}
\begin{proof}
Suppose that $E(G)=E_1\cup E_2$ is a partition of $E(G)$ such that $G_1:=G[E_1]$ and $G_2=G[E_2]$ satisfy~\eqref{no-cm}.

By Claim~\ref{compl}, there is some $i \in \{1,2\}$ such that $\alpha'_*(G_i) = \alpha'(G_i)$.
Pick such an $i$.

Let $(A,C,D)$ be the  GE-decomposition of $G_i$; let $D_0 = C$, $a = |A|$, and let $D_1, D_2, \dots, D_k$ be the components of $G_i[D]$.

We have $N = |V(G)| = |V(G_i)| \ge 2x_1 + x_2 - 1 \ge 2x_i$, and yet by assumption~\eqref{no-cm}, $\alpha'(G_i) < x_i$. Therefore every  maximum matching in $G_i$ leaves at least two vertices uncovered.  Since by Theorem~\ref{GE},  the number of uncovered vertices is $k-a$, this yields $k \ge 2$.

We want to show that $G_i - A$ actually has at least $3$ components. Since $k\ge 2$, $D_1$ and $D_2$ are two of them. If $C = D_0 \ne \emptyset$, then it is a third component of $G_i-A$; if $A \ne \emptyset$, then $k \ge a+2 \ge 3$. If $A = C = \emptyset$ and $k=2$, then $D_1$ and $D_2$ are components of $G_i$ as well. By assumption, $\alpha'_*(G_i) = \alpha'(G_i)$, so $D_1$ and $D_2$ cannot both be nontrivial components.

This leaves the possibility that $D_2$ is an isolated vertex of $G_i$ and $D_1$ is the rest of $V(G)$, which we also will rule out. In this case, by Theorem~\ref{GE}, a maximum matching in $G_i$ covers all vertices of $D_1$ except for one; we have 
\[
	\alpha'_*(G_i) = \frac N2 - 1 \ge \frac{2x_1 + x_2 - 1}{2} - 1 \ge x_i + \frac{x_{3-i}-3}{2}.
\]
But by~\eqref{no-cm}, $\alpha'_*(G_i) \le x_i - 1$, which means $\frac{x_{3-i}-3}{2} \le -1$, or $x_{3-i} \le 1$. By~\eqref{2conn-2}, the degree of the single vertex in $D_2$ is at least $N-n_1\geq x_1 + x_2 - 1 \ge 1$, and it is isolated in $G_i$; therefore $\alpha'_*(G_{3-i}) \ge 1 \ge x_{3-i}$, violating~\eqref{no-cm}. Therefore $G_i - A$ has at least three components.

Let $Q$ be the set of edges in $G_{3-i}$ that are either incident to $A$ or else have both endpoints in the same $D_i$ (including $D_0$). Modify the partition $E_1 \cup E_2$ by removing all edges of $Q$ from $E_{3-i}$ and adding them to $E_i$ instead; let $E_1' \cup E_2'$ be the resulting partition, with $G_1' = G[E_1']$ and $G_2' = G[E_2']$. The same  GE-decomposition $(A,C,D)$ witnesses that $\alpha'(G'_i) = \alpha'(G_i) = \alpha'_{*}(G_i) < x_i$; meanwhile, $G_{3-i}'$ is a subgraph of $G_{3-i}$, so $\alpha'_*(G_{3-i}') \le \alpha'(G_{3-i}) < x_{3-i}$. Therefore the resulting partition still satisfies~\eqref{no-cm}.

Next, we show that $G_{3-i}'$ has at most one nontrivial component: equivalently, that $\alpha'_*(G'_{3-i}) = \alpha'(G_{3-i})$. Suppose for the sake of contradiction that $G_{3-i}'$ has at least two nontrivial components, say $H_1$ and $H_2$. Let $u_1 u_2 \in E(H_1)$ and $v_1v_2 \in E(H_2)$.

We may rename the parts of $G$ so that $u_1\in V_1$ and $u_2\in V_2$. Suppose $u_1\in D_j$ and $u_2\in D_{j'}$. By the definition of $Q$, $j'\neq j$. So, if $v_1\notin V_1\cup V_2$ or $v_1\notin D_j\cup D_{j'}$, then $v_1u_1\in E(G'_{3-i})$ or
$v_1u_2\in E(G'_{3-i})$, and hence $H_2=H_1$. The same holds for $v_2$. Thus, since $v_1v_2\in E(G'_{3-i})$, we may assume that $v_1\in V_1\cap D_{j'}$ and
$v_2\in V_2\cap D_{j}$. We proved earlier that $G_i - A$ has at least three components; therefore we can choose $D_{j''} \ne D_j, D_{j'}$ with a vertex $w \in D_{j''}$. By the symmetry between $V_1$ and $V_2$, we may assume $w\notin V_1$. Then
$w$ is adjacent in $G'_{3-i}$  with both $u_1$ and $v_1$, a contradiction.

The resulting partition $E_1' \cup E_2'$ satisfies $\alpha'_*(G'_1) = \alpha'(G'_1)$ and $\alpha'_*(G'_2) = \alpha'(G'_2)$. The second condition of Claim~\ref{best-coloring} also holds if we had $i=1$ in the proof above. If we had $i=2$, then we may repeat this procedure with $i=1$, finding a third partition $E_1'' \cup E_2''$. This still satisfies $\alpha'_*(G''_1) = \alpha'(G''_1)$ and $\alpha'_*(G''_2) = \alpha'(G''_2)$, but now the Gallai--Edmonds partition of $G''_1$ has the properties we want, proving the claim.
\end{proof}

\subsection{Completing the proof of Theorem~\ref{2conn}}

From now on, we assume that the partition $E_1 \cup E_2$ satisfies the conditions guaranteed by Claim~\ref{best-coloring}. Let $(A,C,D)$ and $D_0, D_1, \dots, D_k$ be as defined in the statement of Claim~\ref{best-coloring}; let $a = |A|$. We can now replace assumption~\eqref{no-cm} by the stronger condition
\begin{equation}
	\alpha'(G_1) < x_1 \text{ and } \alpha'(G_2) < x_2. \label{no-matching}
\end{equation}
The following claim allows us to gradually grow a monochromatic connected matching $R$.

\begin{claim}\label{r-matching}
Let $R$ be a matching in $G_2 - A$. Assume that  $I \ne \emptyset$ is a set of isolated vertices in $G_1 - A$, with $I \cap V(R) = \emptyset$ and $A \cup I \cup V(R) \ne V(G)$.
Suppose that $R$ cannot be made larger by either of the following operations:
\begin{itemize}
\item[(a)] Adding an edge of $G_2$ which has one endpoint in $I$ and the other outside $A \cup I \cup V(R)$.
\item[(b)] Replacing an edge $e \in R$ with two edges $e', e'' \in E(G_2 - A)$ such that $e \subset e' \cup e''$ and $e' \cup e''$ has  one vertex in $I$ and
one in $V(G)-A-R-I$.
\end{itemize}
Then $G$ violates~\eqref{no-matching}.
\end{claim}
\begin{proof}
Let $u$ be a vertex of $G$ outside $A \cup I \cup V(R)$ and let $v \in I$. Since $v$ is an isolated vertex in $G_1 - A$, $uv$ cannot be an edge of $G_1$; by the maximality of $R$, $uv$ cannot be an edge of $G_2$. Therefore  some part $V_i$ of $G$ contains both $u$ and $v$.

Next, we show that 
\begin{equation}
	 \text{\em every edge of $R$ has one endpoint in $V_i$.  } \label{endpoint}
\end{equation}
 Suppose not; let $w_1w_2 \in R$ be an edge with $w_1, w_2 \notin V_i$. Note that $uw_1, uw_2, vw_1, vw_2$ are all edges of $G$. Since $w_1w_2 \in E_2$ and $G_2[D_j]$ is empty for $j = 0,1, \ldots, k$, $w_1$ and $w_2$ cannot be in the same component of $G_1 - A$. Therefore $uw_1, uw_2$ cannot both be in $E_1$; without loss of generality, $uw_1 \in E_2$. Since $v$ is isolated  in $G_1-A$, the edge $w_1w_2 \in R$ can be replaced by the edges $uw_1, vw_2 \in E_2$, violating the maximality of $R$. This proves~\eqref{endpoint}.

By~\eqref{2conn-2}, $v$ has at least $x_1 + x_2 - 1$ neighbors in $G$, so it has at least $(x_1 + x_2 - 1) - a$ neighbors in $G - A$. Since $v$ is an isolated vertex in $G_1- A$, these are all neighbors of $v$ in $G_2$. By the maximality of $V(R)$ (operation (a)), they all are in $V(R)$, and by~\eqref{endpoint}, they are all in different edges of $R$.

Therefore $|R| \ge (x_1 + x_2 - 1) - a$. If $|R| \ge x_2$, then $\alpha'(G_2) \ge x_2$, violating~\eqref{no-matching}. If not, then $(x_1 + x_2 - 1) - a \le x_2 - 1$, so $a \ge x_1$. By Theorem~\ref{GE}, there is a matching in $G_1$ saturating $A$; therefore $\alpha'(G_1) \ge x_1$, again violating~\eqref{no-matching}.
\end{proof}

We consider two cases; in each, we construct the pair $(I,R)$ of Claim~\ref{r-matching} and arrive at a contradiction.

\medskip\noindent\textbf{Case 1:}
$G_2 - A$ has no matching that covers all vertices which are not isolated in $G_1 - A$.

In this case, let $D_1, D_2, \dots, D_r$ be the components of $G_1[D]$ with at least $3$ vertices. For each of these components, we pick a leaf vertex $u_i$ of a spanning tree of $G_1[D_i]$. Since $G_1[D_i]-u_i$ is still connected, there is an edge $e_i \in G_1[D_i] - u_i$. At least one endpoint of $e_i$ is a vertex $v_i$ not in the same part of $G$ as $u_{i+1}$, and is therefore adjacent to $u_{i+1}$ in $G_2$.

To begin, let $R_0$ be the set of the $r-1$ edges $u_{i+1}v_{i}$ found in this way, when $r>0$, and the empty set otherwise. If $I_0$ is the set of all isolated vertices in $G_1[D]$, then $|I_0| = k-r$, and therefore $|I_0| + |R_0| \ge k-1$. 

Now build $I$ and $R$ by the following procedure. Start with $I=I_0$ and $R=R_0$. Whenever an edge (in $G_2$) connects $I$ to $V(G) - (A \cup I \cup V(R))$, add it to $R$ and remove its endpoint from $I$. Whenever we can replace an edge $e \in R$ with two other edges $e', e''$ such that $e \subset e' \cup e''$ and $e' \cup e''$ has exactly one vertex in $I$, do so, and remove from $I$ the vertex contained in $e' \cup e''$. Once this process is complete, $R$ satisfies the maximality conditions of Claim~\ref{r-matching}.

In this process, $|I|+|R|$ never changes. Therefore $|I| + |R| \ge k-1$ at the end of this procedure. 

By~\eqref{no-matching}, $|R| \le \alpha'(G_2) \le x_2 - 1$; therefore $|I| \ge k-1-|R| \ge k - x_2$.

Theorem~\ref{GE} guarantees that $\alpha'(G_1) = \frac{N - (k-a)}{2} \ge \frac{N-k}{2}$. By~\eqref{no-matching}, $\alpha'(G_1) \le x_1 - 1$, so we have
\[
	x_1 - 1 \ge \frac{N - k}{2} \ge \frac{(2x_1 + x_2 - 1) - k}{2} \implies 2x_1 - 2 \ge 2x_1 + x_2 - k - 1 \implies k - x_2 \ge 1.
\]
Therefore $|I| \ge k-x_2 \ge 1$, so $I$ is nonempty.

Moreover, $A \cup I \cup V(R) \ne V(G)$, since by  the case, $R$ does not cover all the non-isolated vertices of $G_1 - A$. Therefore Claim~\ref{r-matching} applies to the pair $(I,R)$, contradicting assumption~\eqref{no-matching}.

\medskip\noindent\textbf{Case 2:}
$G_2 - A$ has a matching that covers all vertices which are not isolated in $G_1 - A$.
Let $R$ be a maximal matching in $G_2 - A$ with this property.
Let $I_0 = V(G) - V(R) - A$. 

By assumption~\eqref{no-matching}, $|V(R)| \le 2\alpha'(G_2) \le 2(x_2 - 1)$, so $|I_0| \ge N - 2(x_2-1) - a$. By~\eqref{2conn-1},
\[
	|I_0| \ge (2x_1 + x_2 - 1) - 2(x_2-1) - a = (x_1-a) + (x_1 - x_2) + 1 \ge x_1 - a +1.
\]
By Theorem~\ref{GE}, there is a matching in $G_1$ saturating $A$. Therefore $a \le \alpha'(G_1) \le x_1 - 1$, and $x_1 - a \ge 1$. Hence $|I_0| \ge 2$.

Choose any $u \in I_0$ and let $I = I_0 - \{u\}$ so indeed $A \cup I \cup V(R) \neq V(G)$. Then Claim~\ref{r-matching} applies to the pair $(I,R)$, with the maximality conditions holding because $R$ is a maximum matching; once again, this contradicts~\eqref{no-matching}.\qed

\section{Connected matchings in 3-edge-colorings (Theorem~\ref{3conn})}

\subsection{ Components of $G_i$}

To prove Theorem~\ref{3conn}, we begin by proving bounds on the sizes of  components in $G_2$ and $G_3$. This is done by applying Theorem~\ref{2conn} to an appropriate subgraph of $G$.

\begin{claim}\label{no-large}
If there is an $i \in \{2,3\}$ such that $G_i$ has no  component of size larger than $x_1 + x_i - 1$, then the conclusion of Theorem~\ref{3conn} holds.
\end{claim}
\begin{proof}
Without loss of generality, say $i=3$. 
For each  component of $G_3$, delete all edges in $G$ between vertices of that component to create a graph $G'$. This graph has a $2$-edge-coloring given by $G_1$ and $G_2$. It satisfies Condition~(\ref{2conn-1}) of Theorem~\ref{2conn} automatically, since $N \ge 2x_1 + x_2 - 1$. Also, no part is larger than $x_1 + x_3 -1$, so
\[
	N-n_i \ge (2x_1 + x_2 + x_3 - 2) - (x_1 + x_3 - 1) = x_1 + x_2 - 1
\]
and $G'$ satisfies Condition~(\ref{2conn-2}). By Theorem~\ref{2conn}, we have $\alpha'_*(G_i) \ge x_i$ 
for some $i \in \{1,2\}$.\end{proof}

\medskip
From now on, we assume that for each $i \in \{2,3\}$, there is a  component in color $i$ on vertex set $S_i \subseteq V(G)$, with $|S_i| \ge x_1 + x_i$.

However, neither $S_2$ nor $S_3$ can be too large.

\begin{claim}\label{too-large}
If there is an $i \in \{2,3\}$ such that $|S_i| \ge x_1 + x_2 + x_3 - 2$, then the conclusion of Theorem~\ref{3conn} holds.
\end{claim}
\begin{proof}
Without loss of generality, say $i=3$. Let $B = V(G) - S_3$.
If $G_3[S_3]$ contains a matching of size $x_3$, then we are done. If not, take the   GE-decomposition $(A, C, D)$ of $G_3[S_3]$.

We build a multipartite graph $G'$, with the inherited $2$-edge-coloring 
 by 
\begin{enumerate}
\item deleting the vertices of $A$ from $G$, and
\item for each component of $G_3[V(G) - A]$, deleting all edges of $G$ inside that component.
\end{enumerate}
We  have $|A| \le x_3 - 1$ because, by Theorem~\ref{GE}, every maximum matching in $G_3[S_3]$ matches each vertex of $A$ to a vertex outside $A$. So $G'$ contains at least $2x_1+x_2+x_3-2 - (x_3-1) = 2x_1+x_2-1$ vertices, satisfying Condition~(\ref{2conn-1}) of Theorem~\ref{2conn}.

If $C_1, \ldots , C_k$ are the  components of $G_3[S_3-A]$, then for each $C_i$ we have $|A|+|C_i| \le 2x_3-1$ because, by Theorem~\ref{GE}, $G_3[C_i]$ is factor-critical and $G_3[S_3]$ has a maximum matching that saturates the vertices in $A\cup C_i$. Therefore $G'-C_i$ contains at least 
\[
	2x_1 + x_2 + x_3 - 2 - (2x_3 - 1) = 2x_1 + x_2 - x_3 - 1 \ge x_1 + x_2 - 1
\]
vertices. 

This verifies Condition~(\ref{2conn-2}) of Theorem~\ref{2conn} for the parts of $G'$ that are contained in $S_3$. It remains to check this condition for parts of $G'$ that are contained in $B$. Since all the vertices of $S_3 - A$ are vertices of $G'$ outside such a part,  the number of such vertices is at least
\[
	|S_3| - |A| \ge (x_1+x_2+x_3-2) - (x_3 - 1) = x_1 + x_2-1.
\]
So Theorem~\ref{2conn} applies to $G'$. Therefore, for some $i \in \{1,2\}$, $\alpha'_*(G_i) \ge \alpha'_*(G'_i) \ge x_i$, and the conclusion of Theorem~\ref{3conn} holds.
\end{proof}

\subsection{Completing the proof of Theorem~\ref{3conn}}

From now on, we assume that the hypothesis of Claim~\ref{too-large} does not hold. Let $\overline{S_i} = V(G) - S_i$. Our assumption implies that $|\overline{S_i}| \ge x_1 + 1$ for both $i \in \{2,3\}$. We can use this to obtain a decomposition of $V(G)$ in which we know the colors of many  edges.

\begin{claim}\label{z-parts}
Theorem~\ref{3conn} holds unless there is a partition $V(G) = Z_0 \cup Z_1 \cup Z_2 \cup Z_3$ such that:
\begin{itemize}
\item All edges of $G[Z_0, Z_1]$ and $G[Z_2, Z_3]$ are in $E_1$.
\item All edges of $G[Z_0, Z_2]$ and $G[Z_1, Z_3]$ are in $E_2$.
\item All edges of $G[Z_0, Z_3]$ and $G[Z_1, Z_2]$ are in $E_3$.
\end{itemize}
Additionally, none of the parts $Z_i$ are empty.
\end{claim}
\begin{proof}
Define the parts as follows: $Z_0 = S_2 \cap S_3$, $Z_1 = \overline{S_2} \cap \overline{S_3}$, $Z_2 = S_2 \cap \overline{S_3}$, and $Z_3 = \overline{S_2} \cap S_3$.

Because $S_2$ and $S_3$ induce  components in $G_2$ and $G_3$ respectively, the edges out of $S_2$ cannot be in $E_2$, and the edges out of $S_3$ cannot be in $E_3$. In particular, this implies that all edges in $G[Z_0, Z_1]$ and $G[Z_2,Z_3]$ are in $E_1$.
The union of the complete bipartite graphs $G[Z_0, Z_1]$ and $G[Z_2,Z_3]$ is a  subgraph of $G_1$. A vertex cover of this bipartite graph has to include either the entire $Z_0$ or the entire $Z_1$, and it has to include either the entire $Z_2$ or the entire $Z_3$. This means a vertex cover contains one of $Z_0 \cup Z_2 = S_2$, or $Z_0 \cup Z_3 = S_3$, or $Z_1 \cup Z_2 =\overline{S_3}$, or $Z_1 \cup Z_3 = \overline{S_2}$. Each of them has size at least $x_1 + 1$ by Claims~\ref{no-large} and~\ref{too-large}.

So this bipartite graph has minimum vertex cover of order at least $x_1 + 1$. Then by Theorem~\ref{konig-egervary}, its maximum matching has size at least $x_1 + 1$. This maximum matching is connected if there is at least one edge from $E_1$ in any of $G[Z_0, Z_2]$, $G[Z_0, Z_3]$, $G[Z_1, Z_2]$, or $G[Z_1, Z_3]$. If this happens, then $\alpha'_*(G_1) \ge x_1+1$ and we obtain the conclusion of Theorem~\ref{3conn}. 

If not, then $G[Z_1,Z_2]$ and $G[Z_0,Z_3]$ cannot contain edges from $E_1$. We already know they cannot contain edges from $E_2$, so they must all be in $E_3$. Similarly, $G[Z_1,Z_3]$ and $G[Z_0,Z_2]$ cannot contain edges from $E_1$ or $E_3$, so they must all be in $E_2$, and the partition has the structure we wanted.

Finally, we check that none of $Z_0, Z_1, Z_2, Z_3$ are empty.

We have $|S_2| + |S_3| = (x_1 + x_2) + (x_1 + x_3) = N+2$, so $|Z_0| = |S_2 \cap S_3| \ge 2$. 

If $Z_1$ were empty, then we would have $|Z_2| = |\overline{S_3}| \ge x_1 + 1$ and $|Z_3| = |\overline{S_2}| \ge x_1 + 1$. In this case, $G[Z_2, Z_3]$ would contain $K_{x_1+1,x_1+1}$, and $\alpha'_*(G_1) \ge x_1$.

The two cases $|Z_2|=0$ and $|Z_3|=0$ are symmetric. If $Z_2$ were empty, then we would have $|Z_0| = |S_2| \ge x_1 + x_2 \ge x_1$ and $|Z_1| = |\overline{S_3}| \ge x_1 + 1$; we would get the same inequalities if $Z_3$ were empty. In either case, $G[Z_0, Z_1]$ would contain $K_{x_1, x_1+1}$, and $\alpha'_*(G_1) \ge x_1$.
\end{proof}
    
Now we complete the proof of Theorem~\ref{3conn}.

\begin{proof}[Proof of Theorem~\ref{3conn}]
Induct on $\min\{x_1,x_2,x_3\}$. The base case is when
\(
	\min\{x_1,x_2,x_3\} = 0,
\)
which holds because we can always find a connected matching of size $0$.

If the theorem holds for all smaller $\min\{x_1,x_2,x_3\}$, then it holds for the triple $(x_1-1,x_2-1,x_3-1)$, so assume this case as the inductive hypothesis.

For the triple $(x_1,x_2,x_3)$, let $G = K_{2x_1 + x_2 + x_3 - 2}$ with a $3$-edge-coloring. If the hypotheses of any of the Claims~\ref{no-large}--\ref{z-parts} hold for $G$, then we are done. Otherwise, $G$ has the decomposition $(Z_0,Z_1,Z_2,Z_3)$ described in Claim~\ref{z-parts}.

Construct a 3-edge-colored subgraph $G'$ of $G$ by deleting a vertex $v_0, v_1, v_2, v_3$ from each of the nonempty sets $Z_0, Z_1, Z_2, Z_3$. $G'$ still has
\[
	N - 4 = 2(x_1-1) + (x_2-1) + (x_3-1) -2
\]
vertices, so the inductive hypothesis applies. We find a connected matching in $G_i'$ of size $x_i - 1$ for some $i$. The vertices of this matching have to be contained in two of the parts $Z_j, Z_k$, with the edges between $Z_j$ and $Z_k$ all having color $i$. So we can add the edge $v_jv_k$ to this matching, getting a connected matching of size $x_i$ in the original $G_i$.
\end{proof}

\section{Stability for 2-edge-colorings (Theorem~\ref{stability})}

\subsection{Proof setup}

 Among counterexamples for  fixed $n,\gamma$ and $\epsilon$ such that $0<\epsilon< 10^{-3}\gamma<10^{-6}$
 and $n>100/\gamma$,
 choose a $2$-edge-colored  $(n,s,\epsilon)$-{suitable}  graph $G$
 with the fewest vertices and modulo this, with the smallest $s$. 
 
 If both \eqref{suit-s1} and \eqref{suit-s2} are strict inequalities, we can delete a vertex from $V_s$ and still have 
 a $2$-edge-colored  $(n,s,\epsilon)$-{suitable}  graph contradicting the minimality of $N$. 
 
 If  $N=3n-1$ and~\eqref{suit-s2} is strict, then $n_1\leq n-1$ and hence $s\geq 3$. Moreover,
 $n_{s-1}+n_s>n$, since otherwise we can consider the $(s-1)$-partite graph obtained from $G$ by deleting all edges between
 $V_{s-1}$ and $V_s$; we have~\eqref{suit-s2} still holds with possibly rearranging the parts according to their size. This also yields that for $s\geq6$, also $n_1+n_2\geq n_3+n_4\geq n_{s-1}+n_s>n$ implying $N>3n$. This contradicts the condition $N=3n-1$. Thus, if $N-n_1>2n-1$, then $N=3n-1$, $s\leq 5$ and $n_1 < n$. 
 
 On the other hand, if $N>3n-1$ and $N-n_1=2n-1$, then $n_1=n_2$, since
 otherwise by deleting a vertex from $V_1$ we get a smaller $(n,s,\epsilon)$-{suitable}  graph. Furthermore, in this case
 $n_1=n_2>(3n-1)-(2n-1)=n$ and hence $n_3+\ldots+n_s<(2n-1)-n=n-1$. So, if $s\geq 4$, then we can replace the parts $V_3,\ldots,V_s$ with
 one part $V'_3=V_3\cup\ldots\cup V_s$ and~\eqref{suit-s2} still holds for the new parts $V_1, V_2, V_3'$. If $s=2$, then $n_1=n_2=2n-1$.\\
  Summarizing, we will replace \eqref{suit-s1} and \eqref{suit-s2} with the following more restrictive conditions:
\begin{equation}
	N \ge 3n-1;\, \text{\em and, if  $N > 3n-1$, then $N - n_1 = 2n-1 \ge n_2 = n_1 > n$ and $s \le 3$.} \label{suit-s1'}\tag{S$1'$}
\end{equation}
\begin{equation}
\text{\em	$N-n_1 \ge 2n-1$;  and if $N-n_1 > 2n-1$, then $N=3n-1, n_1 < n, s \le 5, n_{s-1} + n_s > n$.} \label{suit-s2'}\tag{S$2'$}
\end{equation}
Conditions \eqref{suit-s1'} and \eqref{suit-s2'} imply
\begin{equation}
	N = \max\{n_1, n\} + 2n-1 \le 4n-2, \text{ and } 2n-1 \ge n_1  \ge \ldots \ge n_{s-1} > n/2.
	\label{suit-s5}\tag{S5}
\end{equation}

We obtain $G'$ by deleting from $G$ the set $\widetilde{V}$ and in the case $|V_s-\widetilde{V}|<4\epsilon n$ also deleting $V_s-\widetilde{V}$. 
Let  $s'=s-1$ if we have deleted $V_s-\widetilde{V}$ and $s'=s$ otherwise.
Let $V':=V(G')$ and $N'=|V'|$. 
By \eqref{suit-s3}
and the construction of $V'$,
$N'>N-5\epsilon n$. For $j\in [s']$, let $V'_j=V_j- \widetilde{V}_j$ and $n'_j=|V'_j|$. We also reorder
$V'_j$ and $n'_j$ so that
\begin{equation}\label{2e'}
n'_1\geq n'_2\geq\ldots\geq n'_{s'}.
\end{equation}
For $i\in [2]$, we let $G'_i:=G_i-\widetilde{V}-V_s$ if $|V_s-\widetilde{V}|<4\epsilon n$, and
$G'_i:=G_i-\widetilde{V}$ otherwise.

By construction,~\eqref{2e'} and~\eqref{suit-s5},  $n'_{s'}\geq 4\epsilon n$. In particular,
\begin{equation}\label{2e''}
\mbox{\em for  $j\in [s']$, every $v\in V'_j$ is adjacent to more than half of $V'_{j'}$ for each $j'\in [s']-\{j\}$.}
\end{equation}

The structure of the  proof resembles that of the proof of Theorem~\ref{2conn}, but everything becomes more complicated.
For example, instead of a simple Claim~\ref{s=2}, we need a two pages Subsection~\ref{stab-bipartite-section}  below considering the case of almost bipartite graphs.

For other cases, we will construct a Gallai--Edmonds decomposition of a large subgraph of one $G'_i$ in Subsection~\ref{stab-ge-section}. The rest of this section will prove three lemmas that construct  a $(68\gamma,i,j)$-bad partition of $V(G)$ in different ways, depending on the structure of the Gallai--Edmonds decomposition.

We will repeatedly use the inequality $\gamma> 1000\epsilon$.

\subsection{Nearly bipartite graphs}
\label{stab-bipartite-section}

Suppose that $G$ is an $(n,s,\epsilon)$-{suitable}  graph satisfying \eqref{suit-s1'}, \eqref{suit-s2'} and \eqref{suit-s5}, and that $s'=2$, i.e., $G'$ is bipartite. This means $|V_3|\leq 4\epsilon n$.
 By~\eqref{suit-s2} and the definition of $G'$, 
 \begin{equation}\label{d11}
 |V_1'|\geq |V'_2|\geq 2n-1-5\epsilon n.
 \end{equation}
  Suppose neither of $G'_1$ and $G'_2$ has a connected matching of size
 at least $(1+\gamma)n$.
  Let $F$ be a largest component over all components in $G'_1$ and $G'_2$.
 By symmetry, we may assume that $F$ is a component of $G'_1$. Let $R$ be the smaller of the sets
 $V'_1-V(F)$ and $V'_2-V(F)$, and let $r=|R|$. For $j=1,2$, let 
 $F_j=V(F)\cap V'_j$.
 
We prove two claims that yield Theorem~\ref{stability} for $s'=2$ in two cases, depending on the size of $R$.

\begin{claim}\label{bip-small-r}
	If $r \leq 2\epsilon n$, then $V$ has a $(8\gamma,2,2)$-bad partition.
\end{claim}
\begin{proof}
Since $F$ is the only nontrivial component of $G'_1-R$, 
\[
	\alpha'(G'_1-R)=\alpha'_*(G'_1-R)\leq \alpha'_*(G'_1) \leq (1+\gamma)n.
\]
Hence  by Theorem~\ref{konig-egervary}, $F$ has a vertex cover $Q$ with $|Q|\leq (1+\gamma)n$. Without loss of generality, $|Q \cap V'_1| \leq |Q \cap V'_2|$.
Let $U_1 = Q \cap F_2$ and let $U_2 = V - V_1 - U_1$. We will show that $(V_1, U_1, U_2)$ is an $(8\gamma,2,2)$-bad partition of $V$.

Before verifying the definition of such a partition, we prove some preliminary properties of $U_1$ and $U_2$.

First, by~\eqref{d11},
\begin{equation}\label{d12}
	|V_1'-Q|\geq(1.5-\frac{\gamma}{2}-5\epsilon)n-1
		\text{ and } 
	|V_2'-Q|\geq 2n-1-5\epsilon n-(1+\gamma)n=(1-\gamma-5\epsilon)n-1. 
\end{equation}
Let $H$ be the bipartite graph $G' - Q - R$. By our choice of $V'$,
\begin{equation}\label{d121} 
\parbox{15cm}{
	\em each vertex of $H$ is adjacent to all but at most $\epsilon n$ vertices in the other part.
}
\end{equation}
Moreover, since $Q$ is a vertex cover in $F$, $H$ contains no edges of $G'_1$, so $H = G'_2 - Q - R$. 

By~\eqref{d12} and $r \leq 2\epsilon n$, we have $|V'_j - Q - R| \geq (1 - \gamma - 7\epsilon)n - 1$ for $j=1,2$, and the degree condition of \eqref{d121} tells us that $H$ is connected. Therefore $\alpha'(H) = \alpha'_*(H)$ and, more generally,
\begin{equation}\label{bip-connected-matching}
	\mbox{\em every matching in $G'_2$ such that each edge meets $V'-Q-R$ is  connected.}
\end{equation}


If we greedily construct a matching of size $(1+\gamma)n$ in $H$ by matching vertices in $V'_2 - Q - R$ for as long as possible,
 by~\eqref{d121} we will construct a matching of size at least $\min\{|V'_1 - Q - R| - \epsilon n, |V'_2 - Q - R|\}$, and by~\eqref{bip-connected-matching}, this matching is connected. From~\eqref{d12}, we see that $|V'_1 - Q - R| - \epsilon n > (1+\gamma)n$; therefore 
\begin{equation}\label{v2-q-r size}
	|V'_2 - Q - R| \leq (1 + \gamma)n.
\end{equation}

We are now ready to verify conditions (i)--(v) of an $(8\gamma,2,2)$-bad partition for $(V_1, U_1, U_2)$, though for convenience we will not check them in order.

\medskip\noindent\textbf{(iv) and (v):} We have $|U_1| \leq |Q| \leq (1+\gamma)n$. Meanwhile, $U_2 \subseteq (V'_2 - Q - R) \cup R \cup \tilde{V} \cup V_3$, so $|U_2| \leq |V'_2 - Q - R| + 2\epsilon n + \epsilon n + 4\epsilon n$. By~\eqref{v2-q-r size}, $|U_2| \leq (1 + \gamma + 7\epsilon)n$. On the other hand, $|U_1| + |U_2| = |V - V_1| \geq 2n - 1$, giving us the lower bounds $|U_1| \geq (1 - \gamma - 7\epsilon)n - 1$ and $|U_2| \geq (1 - \gamma)n - 1$.

\medskip\noindent\textbf{(iii):} By~\eqref{d11}, $|V_1| \geq 2n-1-5\epsilon n$.

\medskip\noindent\textbf{(ii):} Since $Q$ is a vertex cover in $F$, every edge in $G_1[V_1, U_2]$ intersects either $Q\cap V_1$ or $V_3\cup \widetilde{V}\cup R$. Since $|V_3\cup \widetilde{V}\cup R|\leq 7\epsilon n$, there are at most $(2n-1)(7\epsilon n) < 14 \epsilon n^2$ edges between $V_1$ and $V_3 \cup \widetilde{V} \cup R$. By (S5), $|Q| \leq (1 + \gamma)n$, and we have checked that $|U_1| \geq (1 - \gamma - 7\epsilon)n - 1 \geq (1 - \gamma - 8\epsilon)n$, we have 
\begin{equation}\label{d13}
	|Q \cap V_1| = |Q| - |U_1| \leq (2\gamma + 8\epsilon) n.
\end{equation}
In particular, there are at most 
$$|U_2| \cdot |Q \cap V_1| \le (1 + \gamma + 7\epsilon)(2 \gamma + 8\epsilon)n^2 \leq 3 \gamma n^2$$ edges between $Q \cap V_1$ and $U_2$.
Therefore $|E(G_1[V_1, U_2])| \leq (3 \gamma + 14\epsilon)n^2$.

\medskip\noindent\textbf{(i):} Suppose for the sake of contradiction that $|E(G_2[V_1,U_1])|> 8\gamma n^2$. By~\eqref{suit-s3} and $|Q|\leq (1+\gamma)n$, $|E(G_2[\widetilde{V}_1\cup R,U_1])|\leq (3\epsilon n) |Q|\leq 3\epsilon(1+\gamma)n^2.$ Similarly, by~\eqref{d13}, 
\[
	|E(G_2[Q \cap V_1,U_1])|\leq |Q \cap V_1|\cdot |Q|\leq 
(2\gamma+13\epsilon)n(1+\gamma)n.
\]
Therefore $|E(G_2[V_1,U_1])|$ can only exceed $8\gamma n^2$ if
\[
	|E(G_2[F_1-Q,U_1])|> (8\gamma-(2\gamma+13\epsilon)(1+\gamma)-3\epsilon(1+\gamma)) n^2>5\gamma n^2.
\]
Since the degree of each vertex in $G[(F_1-Q)\cup U_1]$ is at most
$\max\{|F_1-Q|,|U_1|\}<2n$, this implies that the size $\beta$ of a minimum vertex cover in $G_2[V_1-Q,U_1]$ is at least $2.5\gamma n$.
Then by Theorem~\ref{konig-egervary}, $G_2[F_1-Q,U_1]$ has a matching of size $\beta\geq 2.5\gamma n$. Let $M_1$ be a matching in $G_2[F_1-Q,U_1]$ with $|M_1|=2.5 \gamma n$. Let $Z_1$ be the set of the endpoints
of the edges in $M_1$ that are in $F_1 -Q$. By~\eqref{d121}, each vertex in $F_2-Q$ has in $G'_2$ at least
$|F_1-Q-Z_1|-\epsilon n $ neighbors in $F_1-Q-Z_1$. By~\eqref{d11} and~\eqref{d13}, the number of neighbors each vertex in $F_2-Q$ has in $G'_2$ is at least
\[
	2n-1-7\epsilon n-(2\gamma+13\epsilon)n-2.5\gamma n-\epsilon n>(2-5\gamma)n.
\]
Thus, $G'_2[F_2-Q, F_1-Q-Z_1]$ has a matching $M_2$ covering $F_2-Q$. By~\eqref{bip-connected-matching}, $M_1\cup M_2$ is a connected matching in $G'_2$. And by~\eqref{d11},
\[
	|M_1\cup M_2|=2.5\gamma n+|F_2-Q|\geq 2.5\gamma n+2n-1-7\epsilon n -(1+\gamma)n>(1+\gamma)n,
\]
a contradiction. Thus, $|E(G_2[V_1,U_1])|\leq 8\gamma n^2$.

Therefore the partition $(V_1,U_1,U_2)$ is  $(8\gamma,2,2)$-bad.
\end{proof}

\begin{claim}\label{bip-large-r}
	If $r > 2\epsilon n$, then $V$ has a $(2\gamma, 1, 1)$-bad partition.
\end{claim}
\begin{proof}
For $j = 1,2$ let $\overline F_j = V'_j - F_j$. We know that
\begin{equation}\label{d16}
	\min\{|\overline F_1|, |\overline F_2|\}\geq r\geq 2\epsilon n.
\end{equation}
Without loss of generality, let $|F_1| \geq |F_2|$. Let $W_1 = V(F) = F_1 \cup F_2$ and $W_2 = V - W_1$. We will show that $(W_1, W_2)$ is a $(2\gamma,1,1)$-bad partition of $V$.

Before verifying the definition of such a partition, we will prove lower bounds on $|F_1|$ and $|F_2|$.

First, any vertex $v \in V_1$ has degree at least $|V_2'| - \epsilon n$ in $G'$, which is at least $(2-6\epsilon)n - 1$ by~\eqref{d11}. Therefore in some $G_i'$, $v$ has degree at least $(1-3\epsilon)n-1$, giving a connected component with $(1-3\epsilon)n$ vertices. Hence $|F| \ge (1-3\epsilon)n$ as well; in particular, $|F_1| \ge (1-3\epsilon) n/2$.

Second, suppose that $|F_2| \leq (1-5\epsilon)n$; in this case, by~\eqref{d11}, 
\[
	|\overline F_2| \geq (2n-1-5\epsilon n) - (1 - 5\epsilon)n = n-1,
\] 
and in particular, $|\overline F_2| > |F_2|$. Then $G_2'[F_1, \overline F_2]$ is connected: each vertex is adjacent to all but $\epsilon n$ vertices on the other side, and both $|F_1|$ and $|\overline F_2|$ are much larger than $2\epsilon n$. Hence  $G'_2$ has a component containing $F_1\cup \overline F_2$, and the size of this component is larger than $|F|$, a contradiction to the choice of $F$. 

Therefore $|F_1| \geq |F_2| > (1-5\epsilon)n$, and we are now ready to verify the conditions of a $(2\gamma,1,1)$-bad partition. Again, we will not check them in order.

\medskip\noindent\textbf{(iii):} We will actually show that $E(G_2[W_1]) = \emptyset$. First, $G'_2$ has a connected component containing $F_1$: each vertex of $F_1$ is adjacent (in $G$, and therefore in $G'_2$) to all but $\epsilon n$ vertices of $\overline F_2$; so since $|\overline F_2| > 2\epsilon n$, any two vertices of $F_1$  have a common neighbor in $\overline F_2$. Similarly, $G'_2$ has a connected component containing $F_2$.

Suppose that $G_2$ has an edge $xy$ with $x \in F_1$ and $y \in F_2$. Then the two components above must be the same component, which contains $F_1 \cup F_2$ as well as some vertices of $\overline F_1, \overline F_2$, contradicting the maximality of $F$.

\medskip\noindent\textbf{(i):} By~\eqref{d11}, the quantity in the upper bound of Condition~(i) for $\lambda=2\gamma$ is at least 
\[
	(1 + 2\gamma)(2n - 1 - 5\epsilon n) \ge (2 + 4\gamma - 5\epsilon + 10\gamma \epsilon) n - 2 > (2 + 3\gamma)n - 2.
\]
We know
\[
	|W_2| \leq N - |F_1| - |F_2| \leq 2(2n-1) - 2(1 - 5\epsilon)n = (2 + 10\epsilon)n - 2< (2 + 3\gamma)n - 2.
\]

If $|F_2| \leq (1+\gamma)n$, then  by~\eqref{d11}, 
\[
	|W_2| \geq |V'_2 - F_2| \geq (2n - 1 - 5\epsilon n) - (1 + \gamma)n  > (1 - 2\gamma)n, 
\]
and the lower bound of (i) also holds. Otherwise, $|F_1| \geq |F_2| > (1+\gamma)n$. We have seen that $E(G_2[W_1]) = \emptyset$, and therefore $G[W_1] = G_1[W_1]$ is a bipartite graph where each vertex is adjacent to all but at most $\epsilon n$ vertices on the other side. A vertex cover of $G_1[W_1]$ must contain either $F_1$ or $F_2$ or all but $\epsilon n$ vertices of both, so it has at least $|F_2|$ vertices. Hence by Theorem~\ref{konig-egervary},
 $G_1[W_1]$  has a matching saturating $F_2$, contradicting our choice of $G$.

\medskip\noindent\textbf{(ii):} For every edge $e$ in $G_1[W_1,W_2]$, one of the endpoints must be in $V_3\cup \widetilde{V}$.
Since $|V_3\cup \widetilde{V}|\leq 5\epsilon n$, $|E(G_1[W_1,W_2])|\leq 5\epsilon n|W_1|\leq 20 \epsilon n^2<2\gamma n^2$.
Therefore the partition $(W_1, W_2)$ is $(2\gamma, 1,1)$-bad.
\end{proof}

One of  the Claims~\ref{bip-small-r} and~\ref{bip-large-r} will always apply, proving Theorem~\ref{stability} for $s'=2$.

\subsection{Constructing the Gallai--Edmonds decomposition}
\label{stab-ge-section}

We will now assume $s'\geq 3$.
For $i\in [2]$, let $C_i$ denote the vertex set of the largest component in $G'_i$ and $c_i=|C_i|$. We begin with a claim which will prepare us to use Theorem~\ref{GE} to find a Gallai--Edmonds decomposition of $G'_1[C_1]$.

\begin{claim}\label{stab-compl}
If $|V'-C_i|\geq 4\epsilon n$, then  $G'_{3-i}$ has only one nontrivial component $D$, and there is some $j\in [s']$ such that $ D\supseteq V'-V'_j$.
In particular, if $|V'-C_i|\geq 4\epsilon n$, then $\alpha'(G'_{3-i})=\alpha'_*(G'_{3-i})$.
\end{claim}

\begin{proof} By symmetry, suppose $|V'-C_1|\geq 4\epsilon n$. We begin by constructing a partition $(X_1, X_2)$ of $V'$ with $C_1 \subseteq X_1$.

If $|C_1|\geq n$, then let  $X_2=V'-C_1$. Otherwise, since $N' - |C_1| \ge 2n-1-5 \epsilon n$ and $|C_1| < n$, we obtain $X_2 \subseteq V'-C_1$ by deleting vertex sets of several components of $G'_1$ so that  $n\leq |V'-X_2|<2n$: If $N' - |C_1| \ge 2n$ then since $|C_1| < n$ we delete components until the inequality holds; otherwise, we pick $V' - C_1$. Let $X_1=V'-X_2$.
In any case,
\begin{equation}\label{7n}
|X_2|\geq 4\epsilon n\;\mbox{ and}\;  |X_1|\geq n.
\end{equation}

\medskip\noindent\textbf{Case 1:}
There are $k\in [2]$ and $j,j'\in [s']$ such that $X_k\subseteq V'_j\cup V'_{j'}$. Suppose
$|V'_j\cap X_k|\geq |V'_{j'}\cap X_k|$. Since $s'\geq 3$, there is $j''\in [s']-\{j,j'\}$. By the case, $V'_{j''}\subseteq X_{3-k}$.
Then each  $v\in X_k$ is non-adjacent in $G'_2$ to fewer than $\epsilon n$ vertices in $V'_{j''}$. Since $|V'_{j''}|\geq 4\epsilon n$, 
 every  two vertices in $X_k$ have a common neighbor in $G'_2$.
 So, $G'_2$ has a component
$D$ containing $X_k$. By~\eqref{7n} and the choice of $j$  such that $|V'_j\cap X_k|\geq |V'_{j'}\cap X_k|$, each vertex  in $V(G'_2)-V'_j$ has a neighbor in $X_k$ and hence belongs to $D$.
So, $V'-D\subset V'_j$ and thus
$\alpha'(G'_{2})=\alpha'_*(G'_{2})$.

\medskip\noindent\textbf{Case 2:}
Case 1 does not hold. Since  $s'\geq 3$ and $|V'_j|\geq 4\epsilon n$ for each $j\in [s']$,
 there are $k\in [2]$ and $j,j'\in [s']$ such that $|X_k\cap V'_j|\geq 2\epsilon n$ and $|X_k\cap V'_{j'}|\geq 2\epsilon n$ by the pigeonhole principle.
 Since $|X_k\cap V'_j|\geq 2\epsilon n$,
 every  two vertices in $X_{3-k}-V'_j$ have a common neighbor in $X_k\cap V'_j$ in  $G'_2$.
 So, $G'_2$ has a component
$D$ containing $X_{3-k}-V'_j$.  Similarly, $G'_2$ has a component
$D'$ containing $X_{3-k}-V'_{j'}$. Since Case 1 does not hold, there is  $v\in X_{3-k}-V'_j-V'_{j'}$. This means $D=D'$ and $D\supset  X_{3-k}$.
By~\eqref{7n}, there is at most one $j''\in [s']$ such that  $|X_{3-k}-V'_{j''}|< \epsilon n$ (maybe $j''\in \{j,j'\}$).
Each vertex  in $X_k-V'_{j''}$ has a neighbor in $X_k$ and hence belongs to $D$.
So, $V(G'_2)-D\subset V'_{j''}$ and thus
$\alpha'(G'_{2})=\alpha'_*(G'_{2})$.
\end{proof}

From now on, we assume $c_1\geq c_2$.  Let $B=V'-C_1$ and $b=|B|=N'-c_1$. 

\begin{claim}\label{big} 
  $b\leq  n'_1/2$. 
\end{claim}
\begin{proof} Suppose $b>n'_1/2$. Then $b>4\epsilon n$, so by Claim~\ref{stab-compl} applied to $G'_2$, 
 there is $j\in [s']$ such that $B\subset V'_j$. 
  Since $V'-V'_j\subseteq C_1$ and $|V(G')-V'_j|\geq 2n-1-5\epsilon n$, every two vertices in $B$ have in $G'_2$ a common neighbor in
$V'-V'_j$, and every two vertices in $V'-V'_j$ have a common neighbor in $B$. Thus $G'_2$ has a component $D$ that includes
$B$ and $V'-V'_j$. So 
$$N'-b=c_1\geq c_2\geq |D|\geq N'-|V_j'-B|\geq N'-n'_1+b.$$
Comparing the first and the last expressions in the inequality, we get $n'_1\geq 2b$.
\end{proof}

We are now ready to apply Theorem~\ref{GE} to get a Gallai--Edmonds decomposition of $G'_1[C_1]$, which we will then extend to a 
partition of $V'$. 

Since by Claim~\ref{big},
\[
	c_1 \geq N'-\frac{n'_1}{2}=\frac{1}{2}(N'+(N'-n'_1))\geq \frac{1}{2}(3n-1-5\epsilon n+2n-1-5\epsilon n)> 2(1+\gamma)n,
\]
and $\alpha'_*(G_1)<(1+\gamma)n$, we conclude that
$G'_1[C_1]$ has no perfect matching. Then there is a partition $C_1=A\cup  C\cup \bigcup_{j=1}^k D_j$ satisfying
 Theorem~\ref{GE}.  Let  $a=|A|$.
 
 If $N'-c_1\geq 4\epsilon n$, then also  $N'-c_2\geq 4\epsilon n$, and  by Claim~\ref{stab-compl}
  each vertex in $B$ is isolated
in $G'_1$.
In this case, we view $V'-A$ as the union $\bigcup_{i=0}^{k'}D'_i$, where $k'=k+b$, $D_0=C$. For $1\leq i\leq k$ we define  $D'_i=D_i$. Additionally, for $k+1\leq i\leq k'$, each $D_i$ is a vertex in $B$. By definition, $D_0$ could be empty.   

If $N'-c_1<4\epsilon n$, then we view $V'-A$ as the union $\bigcup_{i=0}^{k'}D'_i$, where $k'=k$, $D'_0=B\cup C$, and
$D'_i=D_i$  for $1\leq i\leq k$. In both cases, we reorder $D'_i$-s so that $|D'_1| \ge \ldots \ge |D'_{k'}|$ and define $d_i:=|D'_i|$ for $i \in [k']$.

Then
  by Theorem~\ref{GE},
 \begin{equation}\label{n10}
\alpha'_*(G'_1)=\alpha'(G'_1[C_1])=\frac{N'-b-k+a}{2}\geq \frac{N'-k'+a}{2}-2\epsilon n.
\end{equation}
 Since $N'\geq 3n-1-5\epsilon n$ and $\alpha'(G'_1)<(1+\gamma)n$,
\eqref{n10} yields a lower bound on $k'$:
 \begin{equation}\label{fn10}
k'\geq a+N'-4\epsilon n-\alpha'_*(G'_1)>a+N'-2(1+\gamma+2\epsilon)n>(1-3\gamma)n+a+2.
\end{equation}

\begin{claim}\label{G2} 
If $G'_2 - A$ is not connected, then the following holds:
\begin{enumerate}
\item[(a)] $a \leq 3\gamma n$;
\item[(b)] $G'_2 - A$ has only one nontrivial component;
\item[(c)] All isolated vertices of $G'_2 - A$ are in the same $V'_j$.
\end{enumerate}
\end{claim}
\begin{proof} We consider two cases.

\medskip\noindent\textbf{Case 1:} $G'_2 - A$ has a vertex $v$ of degree less than $2\epsilon n$.
Suppose that $v \in V'_j$. Then $v$ is adjacent in $G'_1$ to all but at most $3\epsilon n$ vertices of $V' - A - V'_j$. Therefore $G'_1 - A$ consists of a large component containing $v$, at most $3\epsilon n$ components consisting of vertices outside $V_j'$ not adjacent to $v$, and at most $n'_j-1$ isolated vertices in $V_j'$. In particular, $k' \le n'_j + 3\epsilon n \le n_j + 3\epsilon n$.

By~\eqref{n10} and $\alpha'(G'_1) \le (1+\gamma)n$, we have $N' - k' + a - 4\epsilon n \le (2 + 2\gamma)n$, giving us an upper bound on $a$:
$$a	\le (2 + 2\gamma)n - N' + k'  + 4\epsilon n \le (2 + 2\gamma)n - (N - 4\epsilon n) + (n_j + 3\epsilon n) + 4 \epsilon n \le (2 + 2\gamma + 10 \epsilon)n - (N - n_j).$$
By~\eqref{suit-s2'}, we conclude that $a \le (2\gamma + 10\epsilon)n + 1 \le 3 \gamma n$, proving~(a).

Moreover, at least $k' - 3\epsilon n - 1$ components of $G'_1 - A$ are isolated vertices in $V'_j$; by~\eqref{fn10}, this number is much bigger than $2 \epsilon n$. Any vertex not in $V'_j \cup A$ is adjacent in $G'_2$ to more than half of these vertices; therefore, all vertices outside $V'_j$ are in the same component of $G'_2 - A$. We conclude that (b) and (c) hold.

\medskip\noindent\textbf{Case 2:} $G'_2 - A$ has minimum degree at least $2\epsilon n$. Similarly to Case 1, (a) holds. In this case, we will show that $G'_2 - A$ is connected. 
Since $N'\leq 4n - 2$, \eqref{fn10} implies
\[
	\frac{N'}{k'} \leq \frac{4n-2}{(1-3\gamma)n + a + 2} \leq \frac{4n}{(1-3\gamma)n} < 5.
\]
Therefore the average size of the components $D'_1, \dots, D'_{k'}$ is less than $5$; in particular, $D'_{k'}$, the smallest of these components, has fewer than $5$ vertices.

Pick $v \in D'_{k'}$ and let $j$ be such that $v \in V'_j$; let $Q$ be the connected component of $G'_2 - A$ containing $v$. Then $Q$ includes all but at most $\epsilon n + 4$ vertices of $V' - V'_j - A$: it can miss only at most $\epsilon n$ vertices not adjacent to $v$ in $G'$, as well as the other vertices of $D'_{k'}$.

By the case, each vertex of $V'_j - A$ has degree at least $2\epsilon n$ in $G'_2 - A$, which must include a vertex of $Q$; therefore $V'_j - A \subseteq Q$. Now $Q$ includes all but at most $\epsilon n + 4$ vertices of $G'_2 - A$. Then again by the case,  it must include all  its vertices.
\end{proof}
 
In the next three subsections, we will prove the following three lemmas that, together, complete the proof of Theorem~\ref{stability}.

\begin{lemma}\label{d1 lemma}
If $a \le (1-3\gamma)n - 1$, then $|D_1'| \ge N' - a - (1 + \gamma + 4\epsilon)n$.
\end{lemma}

\begin{lemma}\label{small a}
If $a \le (1-3\gamma)n - 1$, then $G'$ has a $(16\gamma,1,1)$-bad partition.
\end{lemma}

\begin{lemma}\label{big a}
If $a \ge (1-3\gamma)n - 1$, then $G'$ has a $(68\gamma,2,1)$-bad or a $(35\gamma,2,2)$-bad partition.
\end{lemma}

\subsection{Proof of Lemma~\ref{d1 lemma}}

We say that a \emph{$G_1'$-balanced split} of $V' - A$ is a partition $(X,Y)$ of $V' - A$ such that $G'_1[X,Y]$ has no edges, and $\min\{|X|, |Y|\} > (1 + \gamma + 4\epsilon)n$.

In this subsection, we will prove Lemma~\ref{d1 lemma} by attempting to construct a $G_1'$-balanced split by splitting $D'_0, D'_1, \dots, D'_{k'}$ between $X$ and $Y$. If this succeeds, we will use the $G_1'$-balanced split either to find a large connected matching, or to show that $A$ (the set of vertices not split between $X$ and $Y$) must be large. The only case in which we will fail to construct a $G_1'$-balanced split is when $D_1'$ is so large that we cannot make the split balanced.

\begin{claim}\label{balanced-split}
	If $a \le (1-3\gamma)n - 1$ and $|D_1'| < N' - a - (1 + \gamma + 4\epsilon)n$, then there exists a $G'_1$-balanced split of $V'-A$.
\end{claim}
\begin{proof}
First, suppose that $(1 + \gamma + 4\epsilon)n < |D'_1| < N' - a - (1 + \gamma + 4\epsilon)n$. In this case, we obtain a $G'_1$-balanced split by taking $X = D'_1$ and $Y = V' - A - D'_1$.

Second, suppose that $|D'_1| \le (1 + \gamma + 4\epsilon)n$. In this case, we construct $X$ and $Y$ step by step. Begin with $X = Y = \emptyset$. For $i=1, 2, \dots, k'$, if $|X| \le |Y|$, set $X = X \cup D'_i$; otherwise, set $Y = Y \cup D'_i$. We stop if $\max\{|X|,|Y|\}$ exceeds $(1+\gamma+4\epsilon)n$; when that happens, we put the remaining sets in the smaller of $X$ and $Y$.

Since 
\[
	N'-a \geq (3n-1-5\epsilon n) -((1-3\gamma)n  - 1) > 2 (1+\gamma+4\epsilon)n,
\]
the algorithm  stops sooner or later. Suppose it stopped after assigning $D'_h$ to $X$ or $Y$. If both $X$ and $Y$ are of size at least $(1+\gamma+4\epsilon)n $, then we have found a $G'_1$-balanced split. So, assume first that $D'_h\subset X$; the argument when $D'_h \subset Y$ is identical.
Then  $|X-D'_h| \le (1+\gamma+4\epsilon)n$ and $|Y| \le (1+\gamma+4\epsilon)n $, but $|X| > (1+\gamma+4\epsilon)n $. 

\medskip\noindent\textbf{Case 1:}
$|D'_h|\leq\frac{ \gamma n}{2}$. Then 

$$N'= |X-D'_h|+|D'_h|+|A|+|Y| < (1+\gamma+4\epsilon)n+ \gamma n/2+(1-3\gamma)n+(1+\gamma+4\epsilon)n $$
$$=(3-0.5\gamma+8\epsilon)n<(3-6\epsilon)n <N',$$
a contradiction.

\medskip\noindent\textbf{Case 2:}
$|D'_h|>\frac{ \gamma n}{2}$. Let $h'$ be the largest index such that $|D'_{h'}|>\frac{ \gamma n}{2}$.
By~\eqref{suit-s5} and the definition of $h'$, $4n>N'-a\geq h'\frac{ \gamma n}{2}$, so
\[
	h\leq h'<4n\cdot \frac{2}{ \gamma n}=\frac{8}{\gamma}<\frac{n}{3}.
\]
By~\eqref{fn10}, $k'\geq (1-3\gamma)n$, so $G'_1-A$ has at least $k'-h'\geq  (1-3\gamma)n-\frac{n}{3}>0.6n$ components of size at most  $\frac{ \gamma n}{2}$.

All these small components were added to $Y$ when the algorithm stops, and so prior to Step $h$, the size of $Y$ was at most $(1 + \gamma + 4\epsilon)n - 0.6n$; in order for $D'_h$ to have been added to $X$, we must have had $|X - D'_h| \le  (1 + \gamma + 4\epsilon)n - 0.6n$ as well, and $|D'_h| \ge 0.6n$.

Note, however, that $h \ge 3$, since in the first two steps we add $D'_1$ to $X$ and $D'_2$ to $Y$, and $|D'_2| \le |D'_1| \le (1 + \gamma + 4\epsilon)n$. In particular, the four sets $D'_1$, $D'_2$, $D'_h$, and $D'_{h'+1} \cup \dots \cup D'_{k'}$ each have size at least $0.6n$, so we obtain a $G'_1$-balanced split by taking $X$ to be the union of any two of them, and taking $Y = V' - A - X$.
\end{proof}

In any $G_1'$-balanced split, if there were many edges of $G'_2$ between $X$ and $Y$, then we would expect to find a large matching in $G'_2$ between them, which would be connected by Claim~\ref{G2}. However, we are assuming that $\alpha'_*(G_2') < (1 + \gamma)n$, so the structure of $G'$ must somehow prevent this. The following claim makes this precise:

\begin{claim}\label{disc2}
For any $G_1'$-balanced split $(X,Y)$, either
\begin{enumerate}
\item[(a)] there is $j \in [s']$ such that $|(X \cup Y) - V'_j| < (1+\gamma+4\epsilon) n$, or

\item[(b)] there are $j, j' \in [s']$ such that $X \cup Y \subseteq V'_j \cup V'_{j'}$ and $G'_2[X \cup Y]$ is disconnected.
\end{enumerate}
\end{claim}
   
\begin{proof}
Suppose that there is a $G'_1$-balanced split $(X,Y)$ for which neither (a) nor (b) holds. In each of the cases below, we will find a connected matching of size $(1+\gamma)n$ in $G'_2$, contradicting our choice of $G$.

\medskip \noindent \textbf{Case 1:}
There is $j\in [s']$ such that $|X-V'_j|<4\epsilon n$ or $|Y-V'_j|<4\epsilon n$. For definiteness, suppose $|X-V'_j|<4\epsilon n$.  Then $|X\cap V'_j|\geq(1+\gamma) n $. Since (a) does not hold,  $|Y-V'_j|>(1+\gamma) n$. Let $F = G'[X \cap V'_j, Y - V'_j]$; since $G'_1[X,Y]$ has no edges, all edges of $F$ come from $G'_2$.

By the construction of $G'$, each vertex of $F$ is adjacent to all but at most $\epsilon n$ vertices in the other part of the bipartition of $F$. Therefore $F$ is connected and, moreover, a vertex cover in $F$ must include either all the vertices in one part, or all but $\epsilon n$ vertices in both parts. Then by Theorem~\ref{konig-egervary}, $F$ has a matching of size
\[
	\min\{|X \cap V'_j|, |Y - V'_j|, |X \cap V'_j| + |Y - V'_j| - 2\epsilon n\} \ge (1 + \gamma)n
\]
and $\alpha'_*(G'_2) \ge (1+\gamma)n$, a contradiction.

\medskip \noindent \textbf{Case 2:}
Case 1 does not hold and there are distinct $j_1,j_2,j_3\in [s']$ such that $X\cap V'_{j_h}\neq \emptyset$  for all $h\in [3]$, say $u_h\in X\cap V'_{j_h}$. 

 Suppose there are $j,j'\in [s']$ such that
 \begin{equation}\label{m15}
 |Y-(V'_j\cup V'_{j'})|<2\epsilon n.
 \end{equation} 
 Since Case 1 does not hold, we have $|Y\cap V'_j|>2\epsilon n$ and $|Y\cap V'_{j'}|>2\epsilon n$. 
 Thus~\eqref{m15} may hold for
 at most
  one pair of $j,j'\in [s']$.
  Hence for all but at most one  pair $(j'',j''')$, any vertices $v_1\in X\cap V'_{j''}$ and
$v_2\in X\cap V'_{j'''}$ have a common neighbor in $Y-(V'_{j''}\cup V'_{j'''})$. In particular, 
\begin{equation}\label{u13}
\mbox{\em the vertices $u_1,u_2$ and $u_3$ are in the same component
of $G'_2[X,Y]$.
}
\end{equation}

Let $F = G'[X, Y] =  G'_2[X,Y]$. By~\eqref{u13}, $F$ has a connected component containing $X$.
 Furthermore, since Case 1 does not hold, each $v \in Y$ has a neighbor in $X$. Thus $F$ is connected and it is enough to show that $\alpha'(F) \geq (1+\gamma) n$.

By Theorem~\ref{konig-egervary}, it is sufficient to prove that
\begin{equation}\label{jb}
\mbox{\em for every $S \subseteq X$, $\quad |N_F(S)| \geq |S|+(1+\gamma) n-|X|$.}
\end{equation}
Let  $\emptyset \neq S\subseteq X$. 
  If $S\subseteq V'_j$ for some $j\in [s']$, then
since (a) does not hold, 
\[
	|N_F(S)|\geq |X \cup Y -V'_j|-|X - S|-\epsilon n \geq (1+\gamma+4\epsilon)n-|X|+|S|-\epsilon n,
\]
and~\eqref{jb} holds. If $S$ intersects two distinct $V'_j$s, then
\[
	|N_F(S)|\geq |Y|-2\epsilon n\geq (1+\gamma+4\epsilon)n-2\epsilon n\geq
(1+\gamma+2\epsilon)n+(|S|-|X|),
\]
and again~\eqref{jb} holds.
  
\medskip \noindent \textbf{Case 3:}
Case 1 does not hold, and there are $j_1,j_2,k_1,k_2 \in [s']$ such that $X \subseteq V'_{j_1} \cup V'_{j_2}$ and $Y \subseteq V'_{k_1} \cup V'_{k_2}$. If $\{j_1,j_2\} \neq \{k_1, k_2\}$, then repeating the argument of Case 2 (with $(j_1, j_2)$ replacing $(j,j')$ and $(k_1,k_2)$ replacing $(j'',j''')$), we again find a connected matching of size at least $(1+\gamma)n$ in $G'_2$. So, suppose $X \cup Y \subseteq V'_{j_1} \cup V'_{j_2}$. Since (b) does not hold, $G'_2[X \cup Y]$ is connected. 

For $h \in [2]$, let $X_h = X \cap V'_{j_h}$ and $Y_h = Y \cap V'_{j_h}$. Since Case 1 does not hold, $|X_h| \ge 4\epsilon n$ and $|Y_h| \ge 4\epsilon n$ for all $h \in [2]$. We can repeat the application of Theorem~\ref{konig-egervary} in Case 1 to show that $G'_2[X_1, Y_2]$ has a matching of size $\min\{|X_1|, |Y_2|\}$ and $G'_2[X_2, Y_1]$ has a matching of size $\min\{|X_2|, |Y_1|\}$. Thus,
\[
	\alpha'_*(G_2[X, Y]) \geq \min\{|X_1| + |X_2|, |X_1| + |Y_1|, |Y_2| + |X_2|, |Y_2| + |Y_1|\}.
\]
We check that all four terms in this minimum are at least $(1+\gamma)n$. This is true for $|X_1| + |X_2| = |X|$ and $|Y_2| + |Y_1| = |Y|$ by the definition of a $G'_1$-balanced split; it is true for $|X_1| + |Y_1| = |(X \cup Y) - V'_{j_2}|$ and $|Y_2| + |X_2| = |(X \cup Y) - V'_{j_1}|$ because (a) is false.
\end{proof}

Finally, we put these two claims together to complete the proof of Lemma~\ref{d1 lemma}.

\begin{proof}[Proof of Lemma~\ref{d1 lemma}]
Assume  that $a \le (1-3\gamma)n - 1$ and $|D_1'| < N' - a - (1 + \gamma + 4\epsilon)n$. Then by Claim~\ref{balanced-split}, we obtain a $G'_1$-balanced split $(X,Y)$ of $V'-A$. Either condition (a) or (b) of Claim~\ref{disc2} must be true for $(X,Y)$.

By~\eqref{suit-s2'}, $|V'-V'_{j_1}|\geq 2n-1-5\epsilon n$. Thus if (a) holds, then
\[
	(2n-1-5\epsilon n)-a\leq |(V'-A)-V'_{j_1}|<(1+\gamma+4\epsilon) n,
\]
and $a> (1-\gamma-9\epsilon)n$, contradicting the condition $a \le (1-3\gamma)n - 1$.

So, suppose (b) holds, in particular, $G'-A$ is bipartite. Since every factor-critical graph is either a singleton or contains an odd cycle, each of $D'_1,\ldots,D'_{k'}$ is a singleton, and only $D_0$ may have more than one vertex. Recall that either $D_0=C$ or $b\leq 4\epsilon n$ and $D_0=B \cup C$.
 Since $G'_1[C]$ has a perfect matching, $C$ is a bipartite graph with equal parts.  So, $|C|\leq 2(1+\gamma)n-a$ and
 $|V'_{j_1}\cap C|=|V'_{j_2}\cap C|\leq (1+\gamma)n-a/2$. By~\eqref{suit-s2'}, for $h\in [2]$,
 
$$|V'_{j_h}-C-A-B| \geq (N'-n'_{j_{3-h}})-|V'_{j_h}\cap C|-a-b \geq 2n-1-((1+\gamma)n-\frac{a}{2})-a-4\epsilon n $$
$$\geq (\frac{1}{2}-\frac{5}{2}\gamma-4\epsilon)n-1 >(\frac{1}{2}-3\gamma)n.$$

Recall that all components of  $G'_1-A-C$ are singletons. This means that for $h\in [2]$, each vertex in $V'_{j_h}-A$
 is adjacent to all but $\epsilon n$ vertices in the set $V'_{j_{3-h}}-C-A-B$ of size at least
$ (\frac{1}{2}-3\gamma)n$. But then $G'_2-A$ is connected, and so does not satisfy~(b).
\end{proof}

\subsection{Small $a$: proof of Lemma~\ref{small a}}

We begin with a general claim about matchings in $G'$.

\begin{claim}\label{stab-disc}
Let $(X,Y)$ be a partition of $V'$ with $0<|X|\leq |Y|$. Write $|X|$ in the form $|X|=n-r$, where $\frac{2n-N'}{2}\leq r\leq n-1$. Then for every $R\subset Y$ with $|R|\leq\min\{r,2r\}+n-1$  such that $G'_1[X, Y-R]$
 has no edges,  the graph $G'_2[X,Y-R]$ has a matching of size at least $|X|-7\epsilon n$.
\end{claim}
\begin{proof}
Let $H = G'_2[X, Y-R]$. By Theorem~\ref{konig-egervary},  it is enough to show that  for every $S\subseteq X$,  
\begin{equation}\label{6e'}
|N_{H}(S)|\geq |S|-7\epsilon n.
\end{equation}

\medskip\noindent\textbf{Case 1:}
$S$ intersects at least two distinct parts of $G'$, say contains vertices $v \in V'_j$ and $w \notin V'_j$ for some $j$.
Then $N_{H}(v)$ contains all but $\epsilon n$ vertices in $(Y-R)-V'_j$, and $N_{H}(w)$ contains all but $\epsilon n$ vertices in $(Y-R)\cap V'_j$. So $|(Y-R)-N_{H}(S)|<2\epsilon n$. But
$$|Y - R| = N' - |X| - |R| \ge (3n-1 - 5\epsilon n) - (n - r) - (\min\{r,2r\}+n-1) = n - 5 \epsilon n + r - \min \{r, 2r\} $$
$$\ge n - r - 5\epsilon n = |Y| - 5\epsilon n \ge |S| - 5\epsilon n,$$
and~\eqref{6e'} holds.

\medskip\noindent\textbf{Case 2:}
$S \subseteq V'_j$ for some $j$. Since $N - n_j \ge 2n-1$ and $N' > N - 5 \epsilon n$, we have $N'-|V'_j| \geq N - 5 \epsilon n - |V_j| \ge 2n-1-5\epsilon n$, and at most $|X-S|$ vertices of $X$ are in $V'-V_j$. So, $Y-R$ has at least $2n-1-5\epsilon n-|X-S|-|R|$ vertices in $V'-V_j$. Let $v\in S$. Since $v$ has at most $\epsilon n$ non-neighbors in $V'-V_j$,
$$|N_H(v)| \geq (2n-1)-5\epsilon n-|X-S|-\epsilon n - |R| \ge |S|- 6\epsilon n + r - \min\{r, 2r\} \ge |S| - 6\epsilon n$$
and~\eqref{6e'} holds.
\end{proof}


\begin{proof}[Proof of Lemma~\ref{small a}]
We assume $a \le (1-3\gamma)n - 1$. By Lemma~\ref{d1 lemma}, $|D'_1| \ge N' - a - (1+\gamma+4\epsilon)n - 1$. 

Since $k'\leq N'-|D'_1|+1$, in our case $k'\leq (1+\gamma+4\epsilon)n +1+1$. This together with~\eqref{fn10} yields
$$ a \leq 2(1+\gamma+2\epsilon)n-N'+k' \leq 2(1+\gamma+2\epsilon)n-3n+1+5\epsilon n+(1+\gamma+4\epsilon)n +2 $$
\begin{equation}\label{m152}
\leq (3\gamma+13\epsilon)n+5 <4\gamma n.
\end{equation}

Let $W_1 = D'_1 \cup A$ and $W_2 = V' - W_1$. We show $(W_1,W_2)$ is a $(16\gamma,1,1)$-bad partition for~$G'$.

\medskip\noindent\textbf{(i):} By~\eqref{fn10}, $|W_2|\geq k'-1>(1-3\gamma)n$. On the other hand, by Lemma~\ref{d1 lemma},
\[
	|W_2|=N'-|D'_1|-a\leq (1+\gamma+4\epsilon)n +1< (1+2\gamma)n.
\]
\medskip\noindent\textbf{(ii):}
Since $D'_1$ has no neighbors in $W_2$ in $G'_1$, \eqref{m152}  yields
 $$|E_{G'_1}[W_1,W_2]|\leq a|W_2|\leq (4\gamma n)|W_2|\leq (4\gamma n) (1+2\gamma)n<5\gamma n^2.$$
\medskip\noindent\textbf{(iii):}
Suppose $\alpha'(G'_2[W_1]) > (4 \gamma+7\epsilon) n$. Let $M_1$ be a matching in $G'_2[W_1]$ with $|M_1|=(4 \gamma+7\epsilon) n$ and let $R=A \cup V(M_1)$. Since $a\leq 4\gamma n$,  $|R|\leq (12\gamma+14\epsilon) n$.

We apply Claim~\ref{stab-disc} with $X = W_2$, $Y = W_1$, and $r=3\gamma n$ (using~\eqref{fn10}).
Since $|R|\leq (12\gamma+14\epsilon) n\leq n-1+r$, graph $G'_2[W_1,W_2] - R$ has a matching $M_2$ of size 
$|W_2|-7\epsilon n\geq k'-1-7\epsilon n$. 

By~\eqref{fn10}, the matching $M_1 \cup M_2$ has size
\[
	|M_1|+|M_2| > (k'-1-7\epsilon n)+(4 \gamma+7\epsilon) n\geq (1-3\gamma)n+4\gamma n=(1+\gamma)n,
\]
and by Claim~\ref{G2}, it is connected, a contradiction. So, $\alpha'(G'_2[W_1]) \le (4 \gamma+7\epsilon) n $.

Hence, by the Erd\H os--Gallai Theorem and~\eqref{suit-s5}, 
$$|E(G'_2[W_1])|\leq (4\gamma+7\epsilon) n |W_1|<16 \gamma n^2.$$
This  completes the proof, since conditions (i)--(iii) of a $(16\gamma,1,1)$-bad partition hold for $(W_1,W_2)$.
\end{proof}

\subsection{Big $a$: proof of Lemma~\ref{big a}}

We start from the following general claim about matchings in $s$-partite graphs.

\begin{prop}\label{tri-min} Let $s\geq 2$ and  $k_1, k_2, \ldots, k_s$ be positive integers.
Let $S=k_1+\ldots+k_s$ and $m=\max\{k_1, k_2, \ldots, k_s\}$.
Let $H$ be obtained from a complete $s$-partite graph $K_{k_1,k_2,\ldots,k_s}$   by deleting some edges in such a way that each vertex loses less than $\epsilon n$ neighbors. 
 Then
 \begin{equation}\label{g(H)}
\alpha'(H) \geq g(H):= \min\{\left\lfloor S/2 \right\rfloor, S-m\}-\epsilon n.
\end{equation}
\end{prop}
\begin{proof}  Let $H$ be a vertex-minimal counterexample to the claim. If $S\leq 2\epsilon n$, then $\frac{S}{2}-\epsilon n\leq 0$, and~\eqref{g(H)} holds trivially, so $S> 2\epsilon n$.
Let the parts of $H$ be $Z_1,\ldots, Z_s$ with $ |Z_i|=k_i$ for $i\in [s]$. Suppose $m=k_1$. Since $S> 2\epsilon n$, either $k_1>\epsilon n$ or $S-k_1>\epsilon n$. In both cases, $H$
has an edge $xy$ connecting $Z_1$ with $V(H)-Z_1$. Let $H'=H-x-y$. 

We claim that $g(H')\geq g(H)-1$. Indeed, $\lfloor \frac{S}{2} \rfloor$ decreases by exactly $1$, and if $S-m$ decreases by $2$, then 
$m$ does not change, which means $k_2=k_1$ and neither $x$ nor $y$ is in $Z_2$. But in this case, since $|\{x,y\}\cap Z_1|=1$,
$S\geq 2m+1$, which yields $S-m\geq \lfloor \frac{S}{2} \rfloor+1=\min\{\lfloor \frac{S}{2} \rfloor, S-m\}+1$, and hence $g(H')\geq g(H)-1$.

So, by the minimality of $H$, $\alpha'(H') \geq g(H')\geq g(H)-1$. Adding edge $xy$ to a maximum matching in $H'$, we complete the proof.
\end{proof}

To prove Lemma~\ref{big a}, we will consider two cases, making different arguments depending on whether we can find a sufficiently large matching in $G'_2[A, V'-A]$. First, however, we prove bounds that are useful in both cases.

By \eqref{fn10} and \eqref{suit-s5},
$$k' \ge N'+a - 2 (1+\gamma+2\epsilon) n  \ge \max\{n_1,n\}+2n-1-9\epsilon n+(1-3\gamma)n - 1- 2 (1+\gamma) n.$$
So,
\begin{equation}\label{V'-A}
k' \ge \max\{n_1,n\}+ n- (5\gamma+9\epsilon)n-2.
\end{equation}
Construct an independent set $I$ in $G'_1-A-D_0$ of size $k'$ by choosing one vertex from 
 each  component of $G'_1-A-D_0$.   Let  $Q = V' - A - I$. Then by~\eqref{suit-s5}, 
$$
|V'-A| \leq \max\{n_1,n\}+2n-1 - a \leq \max\{ n_1,n\} + 2n-1-( (1-3\gamma)n - 1),
$$
and thus  by~\eqref{V'-A},
$$|Q| \le N'-a-k' \leq    \max\{n_1,n\}+2n-1-( (1-3\gamma)n - 1)-( \max\{n_1,n\}+ n- (5\gamma+9\epsilon)n-2).$$
Hence
\begin{equation}\label{Q}
|Q|\leq   8 \gamma n+9\epsilon n + 2<9  \gamma n.
\end{equation}

\begin{claim}\label{large-a-case-1}
	If $\alpha'(G'_2[A, V'-A]) \le 8\gamma n$, then $V$ has a $(68\gamma, 2,1)$-bad partition.
\end{claim}
\begin{proof}
Since  $G'_2[A,V'-A]$ is bipartite, by Theorem~\ref{konig-egervary}, it has  a vertex cover $X$ with $|X|\leq 8 \gamma n$. Let $W_2 = A - X$, and $W_1 = V' - W_2$.
We will show that
$(W_1,W_2)$ is a $(68\gamma, 2,1)$-bad partition for $G'$ by checking all conditions.

\medskip\noindent\textbf{(i):} Since $a \ge (1-3\gamma)n - 1$ and $|X| \le 8 \gamma n$, 
 $$|W_2| = |A - X| \ge a - |X| \ge (1-3\gamma n) - 1 - 8 \gamma n \ge (1-12\gamma)n.$$ 
 On the other hand, $|W_2| = |A - X| \le a \le (1+\gamma)n.$
 
\medskip\noindent\textbf{(ii):}
 Since $X$ is a vertex cover in $G'_2[A, V'-A]$, $G'_2$ has  no edge in $G_2$ between $W_2-X=W_2$ and $W_1 - X$. 
 Thus, 
 $$|E(G'_2[W_1,W_2])| \le |X\cap W_1| \cdot |W_2| \le 8 \gamma n \cdot a <16 \gamma n^2.$$
\medskip\noindent\textbf{(iii):}
Since $I$ is an independent set in $G'_1$, by~\eqref{Q}, 
$$|E(G'_1[W_1])| \le |Q \cup (A \cap X)| \cdot |W_1| \le 17 \gamma n  N' \le 68 \gamma  n^2. $$

This completes the proof.
\end{proof}

\begin{claim}\label{large-a-case-2}
	If $\alpha'(G'_2[A, V'-A]) \ge 8\gamma n$, then $V$ has a $(35\gamma,1,2)$-bad partition.
\end{claim}

\begin{proof}
Let $X$ be a matching of size $8 \gamma n$ in $G'_2[A, V'-A]$. 

\medskip\noindent\textbf{Step 1:}
Our first step is to prove some preliminary facts about $X$.
Since $|I|=k'$, by~\eqref{V'-A}, 
\begin{equation}\label{I-V(X)}
|I-V(X)| \ge  \max\{n_1,n\}+ n- (5\gamma+9\epsilon)n-2 - 8 \gamma n = \max\{n_1,n\}+ (1-13 \gamma-9\epsilon )n - 2.
\end{equation}
Let $R$ be a matching of size $\alpha'({G'_2}[I-V(X)])$ in $I-V(X)$ in $G'_2$. Since $a>3\gamma n$, by Claim~\ref{G2}, $G'_2-A$ is connected, and hence $R \cup X$ is a connected matching in $G'_2$. Since $\alpha'_*(G'_2) < (1+\gamma) n$,  
$$|R| + |X| = \alpha'({G'_2}[I-V(X))]) + 8 \gamma n < (1+\gamma)n.$$ Therefore, 
\begin{equation}\label{I-V(X)'}
\alpha'({G'_2}[I-V(X)]) < (1 - 7 \gamma)n.
\end{equation}

For all $j \in [s']$, let $X_j = V'_j \cap V(X) \cap I,$ and $Y_j = V'_j \cap I - V(X)$ for $j \in [s']$. Let $h \in [s']$ be such that $|Y_h| = \max\{|Y_j| : j \in [s']\}$. By Proposition~\ref{tri-min}, 
\begin{equation}\label{NM}
\alpha'(G_2[I-V(X])\geq \min\left\{\left\lfloor \frac{|I - V(X)|}{2} \right\rfloor,  |I-V(X)-Y_{h}|\right\}-\epsilon n.
\end{equation}
Since by~\eqref{I-V(X)} and~\eqref{I-V(X)'},
$$\left\lfloor \frac{|I - V(X)|}{2} \right\rfloor \ge \left\lfloor  \frac{k'-8\gamma n}{2}\right\rfloor \geq n-1- \frac{(13\gamma+9\epsilon)n}{2} > (1-7\gamma+2\epsilon)n \ge \alpha'(G_2[I-V(X)])+2\epsilon n,$$
\eqref{I-V(X)'} and~\eqref{NM}  yield
 \begin{equation}\label{i-v-y}
  |I-V(X)-Y_{h}|-2\epsilon n \leq \alpha'(G_2[I-V(X)]) \le (1-7 \gamma)n.
  \end{equation}
   Again by~\eqref{I-V(X)},
\begin{equation}\label{Y_1}
|Y_{h}|  \ge \max\{n_1,n\}+ (1-13 \gamma-9\epsilon )n - 2
 - (1 - 7 \gamma)n \ge \max\{n_1,n\} - 6.5\gamma n .
\end{equation}

 \noindent By \eqref{Y_1}, we have 
 \begin{equation}\label{Y_1'}
 |A\cap V'_{h}| \le |V'_{h}| - |Y_{h}| \le n_1 - (n_1 - 6.5 \gamma n ) = 6.5 \gamma n.
  \end{equation}
  
\medskip\noindent\textbf{Step 2:}
Our second step is to modify the matching $X$ to satisfy the following condition:
\begin{equation}\label{improved-x-property}
	\alpha'(G'_2[A - V'_{h}, V'_h - A]) = |X_{h}| \text{ and } \alpha'(G'_2[A - V'_h, V'_h - A]) \le 7\gamma n.
\end{equation}  
Fix a maximum matching $S$ in $G'_2[A - V'_h, V'_h - A]$.

Let $M_j$ be the subset of matching edges of $X$ with an endpoint in $X_j$. By definition, $|M_{h}|=|X_{h}| \le |S|$. For as long as $|X_h| < |S|$, we repeat the following procedure to increase $|X_h|$.

Each component of $S\cup M_{h}$ is a path or a cycle. Since $|S|>|M_{h}|$, there is a component $C$ (a path) of $S\cup M_{h}$ with one more edge in $S$ than in $M_{h}$. Say the endpoints of $C$ are $w_1$ and $w_2$. Then we can assume $w_1 \in Y_h$ and $w_2 \in A$. There are two cases:
\begin{itemize}
\item If $w_2$ is incident with an edge $e \in X-M_{h} $, then we switch the edges in $C$ (if an edge was originally in $S$  then now it is in $M_{h}$ and vice versa) and delete $e$ from $X$. 
\item If $w_2$ is not incident with any matching edge in $X-M_{h}$, then we switch the edges in $C$ and delete any  edge $e \in X-M_{h}$. 
\end{itemize}
In both cases, we obtain a new matching $X'$ with $|X'| = |X|$ and $|X'_{h}| = |X_{h}|+1$. Note that~\eqref{Y_1} still works for $X'$ and by~\eqref{Y_1'}, 
   \begin{equation}\label{Y_1''}
|X'_{h}| \le |V'_{h}| - |Y'_{h}| <7 \gamma n .
 \end{equation}
   Thus repeating the procedure, on every step we increase $|X'_{h}|$, but preserve~\eqref{Y_1''}. 
 Eventually we construct a matching $X''$
   with $|X''_{h}| = \alpha'(G_2[A - V'_h, V'_h - A])<7\gamma n$.

\medskip\noindent\textbf{Step 3:} We are finally ready to construct the partition that proves Claim~\ref{large-a-case-2}. Let $U_1 = A - V_h$ and $U_2 = V(G) - A - V_h$. We now show that $(V_{h},U_1,U_2)$ is a $(35\gamma,1,2)$-bad partition by checking conditions (i)--(v) in the definition.

\medskip\noindent\textbf{(i):}
 Since by~\eqref{Q} and~\eqref{i-v-y},
   \begin{equation}\label{n16}
 |U_2| \le |I-V(X)-Y_{h}|
   + |Q| + |X| \le (1-7\gamma +2\epsilon)n + 9 \gamma n + 8 \gamma n\leq (1+ 10\gamma+2\epsilon)n,
\end{equation}
$$\hspace{-3.5cm}\text{ we have }\quad \quad \quad \quad |E(G'_1[V'_{h}, U_2])| \le |A\cap V_{h}| \cdot |U_2|+ |Q| \cdot |U_2| + |Q| \cdot |Y_{h}| $$
$$\le (6.5\gamma n)(1+10\gamma+2\epsilon)n  + 9\gamma n (1+10\gamma+2\epsilon)n +  9\gamma n(2n-1) \le 35 \gamma n^2.$$
\medskip\noindent\textbf{(ii):}
By~\eqref{Y_1'} and~\eqref{improved-x-property}, 
$$|E(G'_2[U_1,V_{h}])| \le 7\gamma n \cdot n_1 + |A\cap V'_{h}| \cdot |U_1| \le 7\gamma n (2n-1)+ 6.5 \gamma n (1+\gamma)n <22 \gamma n^2.$$
\medskip\noindent\textbf{(iii):}
By~\eqref{Y_1}, $|V'_{h}| \ge |Y_{h}| \ge (1-6.5\gamma)n.$

\medskip\noindent\textbf{(iv):}
 Since $a \ge (1-3\gamma)n - 1$, by~\eqref{Y_1'},  
 $$ (1-10\gamma)n-1 \le (1-3\gamma)n - 1 - 6.5 \gamma n  \le a-|A\cap V_{h}|= |U_1| \le a \le (1+\gamma)n.$$  
\medskip\noindent\textbf{(v):}
 By~\eqref{Y_1}, $$|U_2| = N' - |V_{h}| - |U_1| \ge (n_1 + 2n - 1 -5\epsilon n)- n_1 - (1+\gamma)n = (1-2\gamma)n.$$
On the other hand, by~\eqref{n16}, $|U_2|  \le (1+11\gamma)n.$
This completes the proof of the claim.
\end{proof}

The hypothesis of either Claim~\ref{large-a-case-1} or Claim~\ref{large-a-case-2} must hold, completing the proof of Lemma~\ref{big a}, which was the final step of proving Theorem~\ref{stability}.

\bigskip\noindent
{\bf \large Acknowledgement.}
We thank Louis DeBiasio for helpful discussions. We also thank  both referees for detailed and helpful comments. Moreover, one of the referees suggested a nicer and shorter proof of Theorem 4.

\end{document}